\documentclass[a4paper]{amsart}

\usepackage{graphicx}
\usepackage{amsthm,amsmath,amsfonts,amssymb,mathrsfs}
\usepackage{lmodern}  % per avere bold+italic
\usepackage{bbm} % per avere il comando \mathbbmss
\usepackage{enumitem}
\usepackage{tikz-cd}  % per i diagrammi commutativi

\newtheorem{definition}{Definition}[section]
\newtheorem{corollary}{Corollary}[section]
\newtheorem{proposition}{Proposition}[section]

\newtheorem{remark}{Remark}[section]
\newtheorem{notation}{Notation}[section]

\newcommand{\df}[1]{\textbf{\textit{#1}}}  % usato per le definizioni
\newcommand{\Tau}{\mathcal{T}} % T maiuscola calligrafica usato come TAU grande per le topologie
\newcommand{\E}{\mathbbmss{E}}
\newcommand{\U}{\mathbbmss{U}}
\newcommand{\PP}{\mathbb{P}} % insieme delle parti
\newcommand{\NN}{\mathbb{N}} % insieme dei numeri naturali
\newcommand{\Ncal}{\mathcal{N}}  % N calligrafico, usato per le famiglie di intorni
\newcommand{\Acal}{\mathcal{A}}  % A calligrafico
\newcommand{\Bcal}{\mathcal{B}}  % B calligrafico
\newcommand{\Scal}{\mathcal{S}}  % S calligrafico
\newcommand{\SSE}[1][\U]{{\mathcal{SS}(#1)}_\E}  % insieme dei soft sets
\newcommand{\SSG}[2]{{\mathcal{SS}(#1)}_{#2}} % insieme generico dei soft sets
\newcommand{\SPE}[1][\U]{{\mathcal{SP}(#1)}_\E}  % insieme dei soft points
\newcommand{\SPG}[2]{{\mathcal{SP}(#1)}_{#2}} % insieme generico dei soft points
\newcommand{\softsubseteq}{\tilde{\subseteq}}  % soft subset
\newcommand{\softequal}{\tilde{=}}  % soft equal
\newcommand{\softin}{\tilde{\in}}  % soft appartenenza di un soft point a un soft set
\newcommand{\softnotequal}{\tilde{\ne}}  % soft not equal (or soft distinct)
\newcommand{\softnotsubseteq}{\tilde{\not\subseteq}}  % not soft subset
\newcommand{\softnotin}{\tilde{\notin}}  % soft non appartenenza di un soft point a un soft set
\newcommand{\nullsoftset}{(\tilde{\emptyset},\E)}  % null soft set
\newcommand{\absolutesoftset}[1][\U]{(\tilde{#1},\E)}  % absolute soft set
\newcommand{\nullsoftsetG}[1]{\left(\tilde{\emptyset},#1\right)}  % generic null soft set
\newcommand{\absolutesoftsetG}[2]{\left(\tilde{#1},#2\right)}  % generic absolute soft set
\newcommand{\softsetminus}{\widetilde{\setminus}}  % soft differenza
\newcommand{\softcup}{\tilde{\cup}}  % soft unione
\newcommand{\softcap}{\tilde{\cap}}  % soft intersezione
\newcommand{\softbigcup}{\widetilde{\bigcup}}  % soft unione generalizzata
\newcommand{\softbigcap}{\widetilde{\bigcap}}  % soft intersezione generalizzata
\newcommand{\softprod}{\widetilde{\prod}}  % soft prodotto cartersiano generalizzato
\newcommand{\softcirc}{\widetilde{\circ}}  % soft composizione di soft mappings
\tikzset{commutative diagrams/.cd, softcircdiagram/.style={start anchor=center,end anchor=center,draw=none}}
\newcommand{\softcl}[2][X]{\operatorname{s-cl}_{#1} #2 }  % soft chiusura senza le parentesi
\newcommand{\softclpar}[2][X]{\operatorname{s-cl}_{#1}\!\!\left( #2 \right)}  % soft chiusura con le parentesi
\newcommand{\rl}[1][Y]{{}^{#1}\!}   % sub soft set (o insieme relativo)
\newcommand{\softhomeomorphic}{\tilde{\thickapprox}}  % soft omemomorfismo

\begin{document}
\title{A Soft Embedding Lemma for Soft Topological Spaces}
\author{Giorgio Nordo}

\maketitle

\begin{abstract}
In 1999, Molodtsov initiated the theory of soft sets as a new mathematical tool
for dealing with uncertainties in many fields of applied sciences.
In 2011, Shabir and Naz introduced and studied the notion of soft topological spaces, also defining
and investigating many new soft properties as generalization of the classical ones.
In this paper, we introduce the notions of soft separation between soft points and soft closed sets in order
to obtain a generalization of the well-known Embedding Lemma to the class of soft topological spaces.
\end{abstract}

%-------------------------------------------------------------------------------------------
\section{Introduction}
%----- soft sets

Almost every branch of sciences and many practical problems in
engineering, economics, computer science, physics, meteorology, statistics, medicine, sociology, etc.
have its own uncertainties and ambiguities because they
depend on the influence of many parameters and,
due to the inadequacy of the existing theories of parameterization in dealing with uncertainties,
it is not always easy to model such a kind of problems
by using classical mathematical methods.
%---
In 1999, Molodtsov \cite{molodtsov} initiated the novel concept
of Soft Sets Theory as a new mathematical tool
and a completely different approach for dealing with uncertainties
while modelling problems in a large class of applied sciences.
%----------

In the past few years, the fundamentals of soft set theory have been studied by many researchers.
Starting from 2002, Maji, Biswas and Roy \cite{maji2002, maji2003} studied the theory of soft sets initiated by
Molodtsov, defining notion as the equality of two soft sets, the subset and super set of a soft set,
the complement of a soft set, the union and the intersection of soft sets,
the null soft set and absolute soft set, and they gave many examples.
In 2005, Pei and Miao \cite{pei} and Chen et al. \cite{chen} improved the work of Maji.
Further contributions to the Soft Sets Theory were given by Yang \cite{yang},
Ali et al. \cite{ali},
Fu \cite{fu},
Qin and Hong \cite{qin},
Sezgin and Atag\"{u}n \cite{sezgin},
Neog and Sut \cite{neog},
Ahmad and Kharal \cite{ahmad2009},
Babitha and Sunil \cite{babitha2010},
Ibrahim and Yosuf \cite{ibrahim},
Singh and Onyeozili \cite{singh},
Feng and Li \cite{feng},
Onyeozili and Gwary \cite{onyeozili},
\c{C}a\u{g}man \cite{cagman2014}.

%----- soft topological spaces
In 2011, Shabir and Naz \cite{shabir} introduced the concept of soft topological spaces,
also defining and investigating the notions of soft closed sets, soft closure,
soft neighborhood, soft subspace and some separation axioms.
Some other properties related to soft topology were studied by
\c{C}a\u{g}man, Karata\c{s} and Enginoglu in \cite{cagman2011}.
In the same year Hussain and Ahmad \cite{hussain} investigated
the properties of soft closed sets, soft neighbourhoods, soft interior, soft exterior
and soft boundary, while
Kharal and Ahmad \cite{kharal} defined the notion of a mapping on soft classes
and studied several properties of images and inverse images.
%----------
The notion of soft interior, soft neighborhood and soft continuity were also
object of study by Zorlutuna, Akdag, Min and Atmaca in \cite{zorlutuna}.
Some other relations between these notions was proved by Ahmad and Hussain in \cite{ahmad}.
The neighbourhood properties of a soft topological space were investigated in 2013
by Nazmul and Samanta \cite{nazmul}.
The class of soft Hausdorff spaces was extensively studied by Varol and Ayg\"{u}n in \cite{varol2013}.
%--
In 2012, Ayg\"{u}no\u{g}lu and Ayg\"{u}n \cite{aygunoglu}
defined and studied the notions of soft continuity and soft product topology.
Some years later, Zorlutuna and \c{C}aku \cite{zorlutuna2015}
gave some new characterizations of soft continuity, soft openness and soft closedness of soft mappings,
also generalizing the Pasting Lemma to the soft topological spaces.
Soft first countable and soft second countable spaces were instead defined and studied by Rong in \cite{rong}.
Furthermore, the notion of soft continuity between soft topological spaces was independently introduced
and investigated by Hazra, Majumdar and Samanta in \cite{hazra}.
%--
Soft connectedness was also studied in 2015 by Al-Khafaj \cite{al-khafaj} and Hussain \cite{hussain2015a}.
In the same year, Das and Samanta \cite{das,das2} introduced and extensively studied the soft metric spaces.
%--
In 2015, Hussain and Ahmad \cite{hussain2015b} redefined and explored several properties of
soft $T_i$ (with $i=0,1,2,3,4$) separation axioms and discuss some soft invariance properties
namely soft topological property and soft hereditary property.
%--
In \cite{xie}, Xie introduced the concept of soft points and proved
that soft sets can be translated into soft points so that they may conveniently dealt
as same as ordinary sets.
%--
In 2016, Tantawy, El-Sheikh and Hamde \cite{tantawy} continued the study of soft $T_i$-spaces
(for $i=0,1,2,3,4,5$) also discussing the hereditary and topological properties for such spaces.
In 2017, Fu, Fu and You \cite{fu2017} investigated some basic properties concerning the soft topological product space.
%---
Further contributions to the theory of soft sets and that of soft topology were added
in 2011, by Min \cite{min},
%--
in 2012, by Janaki \cite{janaki},
and by Varol, Shostak and Ayg\"{u}n \cite{varol},
%--
in 2013 and 2014, by Peyghan, Samadi and Tayebi \cite{peyghan},
by Wardowski \cite{wardowski},
by Nazmul and Samanta \cite{nazmul2014},
by Peyghan \cite{peyghan2014},
and by Georgiou, Megaritis and Petropoulos \cite{georgiou2013,georgiou2014},
%--
in 2015 by Ulu\c{c}ay, \c{S}ahin, Olgun and Kili\c{c}man \cite{ulucay},
and by Shi and Pang \cite{shi},
%--
in 2016 by Wadkar, Bhardwaj, Mishra and Singh \cite{wadkar},
by Matejdes \cite{matejdes2016},
and by Fu and Fu \cite{fu2017},
%--
in 2017 by Bdaiwi \cite{bdaiwi},
and, more recently, by Bayramov and Aras \cite{bayramov},
and by Nordo \cite{nordo2018,nordo2019}.

%----- in this paper..
In the present paper we will present the notions of family of soft mappings
soft separating soft points and soft points from soft closed sets
in order to give a generalization of the well-known Embedding Lemma for soft topological spaces.

%-------------------------------------------------------------------------------------------
\section{Preliminaries}
%--- soft sets
In this section we present some basic definitions and results on soft sets and suitably exemplify them.
Terms and undefined concepts are used as in \cite{engelking}. %----- ESISTE UNA FORMA MIGLIORE?

%----------------------------

%--- definition of soft set
\begin{definition}{\rm\cite{molodtsov}}
\label{def:softset}
Let $\U$ be an initial universe set and $\E$ be a nonempty set of parameters (or abstract attributes)
under consideration with respect to $\U$ and $A\subseteq \E$,
we say that a pair $(F,A)$ is a \df{soft set} over $\U$
if $F$ is a set-valued mapping $F: A \to \PP(\U)$
which maps every parameter $e \in A$ to a subset $F(e)$ of $\U$.
\end{definition}

In other words, a soft set is not a real (crisp) set
but a parameterized family $\left\{ F(e) \right\}_{e\in A}$ of subsets of the universe $\U$.
For every parameter $e \in A$, $F(e)$ may be considered as the set of \textit{$e$-approximate elements}
of the soft set $(F,A)$.

\begin{remark}
\label{rem:sameparameter}
In 2010, Ma, Yang and Hu \cite{ma} proved that every soft set $(F,A)$ is equivalent
to the soft set $(F,\E)$ related to the whole set of parameters $\E$,
simply considering empty every approximations of parameters which are missing in $A$,
that is extending in a trivial way its set-valued mapping,
i.e. setting $F(e)=\emptyset$, for every $e \in \E \setminus A$.
\\
For such a reason, in this paper we can consider all the soft sets over the same parameter set $\E$
as in \cite{chiney} and we will redefine all the basic operations and relations
between soft sets originally introduced in \cite{maji2002,maji2003,molodtsov} as in \cite{nazmul},
that is by considering the same parameter set.
\end{remark}

\begin{definition}{\rm\cite{zorlutuna}}
\label{def:setofsoftsets}
The set of all the soft sets over a universe $\U$ with respect to a set of parameters $\E$
will be denoted by $\SSE$.
\end{definition}

%-- definition of soft subset
\begin{definition}{\rm\cite{nazmul}}
\label{def:softsubset}
Let $(F,\E),(G,\E) \in \SSE$ be two soft sets over a common universe $\U$
and a common set of parameters $\E$,
we say that $(F,\E)$ is a \df{soft subset} of $(G,\E)$ and we write
$(F,\E) \softsubseteq (G,\E)$
if $F(e)\subseteq G(e)$ for every $e \in \E$.
\end{definition}

%-- definition of soft equality
\begin{definition}{\rm\cite{nazmul}}%{\rm\cite{ali}}
\label{def:softequal}
Let $(F,\E),(G,\E) \in \SSE$ be two soft sets over a common universe $\U$, we say that
$(F,\E)$ and $(G,\E)$ are \df{soft equal} and we write $(F,\E) \softequal (G,\E)$
if $(F,\E) \softsubseteq (G,\E)$ and $(G,\E) \softsubseteq (F,\E)$.
\end{definition}

%-- equivalence for soft equality
\begin{remark}
\label{rem:softequal}
If $(F,\E),(G,\E) \in \SSE$ are two soft sets over $\U$,
it is a trivial matter to note that $(F,\E) \softequal (G,\E)$ if and only if
it results $F(e)=G(e)$ for every $e \in \E$.
\end{remark}

%-- definition of null soft set
\begin{definition}{\rm\cite{nazmul}}
\label{def:nullsoftset}
A soft set $(F,\E)$ over a universe $\U$ is said to be the \df{null soft set}
and it is denoted by $\nullsoftset$ if $F(e) = \emptyset$ for every $e \in \E$.
\end{definition}

%-- definition of absolute soft set
\begin{definition}{\rm\cite{nazmul}}
\label{def:absolutesoftset}
A soft set $(F,\E) \in \SSE$ over a universe $\U$ is said to be the \df{absolute soft set}
and it is denoted by $\absolutesoftset$
%(and sometimes simply by $\tilde{U}$)
if $F(e) = \U$ for every $e \in \E$.
\end{definition}

%-- definition of constant soft set
\begin{definition}%{\rm\cite{hussain}}
\label{def:constantsoftset}
Let $(F,\E) \in \SSE$ be a soft set over a universe $\U$
and $V$ be a nonempty subset of $U$,
the \df{constant soft set} of $V$, denoted by $(\tilde{V},\E)$)
(or, sometimes, by $\tilde{V}$), is the soft set $(\underbar{V},\E)$,
where $\underbar{V}: \E \to \PP(\U)$ is the constant set-valued mapping
defined by $\underbar{V}(e) = V$, for every $e \in \E$.
\end{definition}

%-- definition of soft complement
\begin{definition}{\rm\cite{nazmul}}
\label{def:softcomplement}
Let $(F,\E) \in \SSE$ be a soft set over a universe $\U$, the \df{soft complement}
(or more exactly the \textit{soft relative complement}) of $(F,\E)$,
denoted by $(F,\E)^\complement$, is the soft set $\left( F^\complement, \E \right)$
where $F^\complement : \E \to \PP(\U)$ is the set-valued mapping
defined by $F^\complement(e) = F(e)^\complement = \U \setminus F(e)$, for every $e \in \E$.
\end{definition}

%-- definition of soft difference
\begin{definition}{\rm\cite{nazmul}}
\label{def:softdifference}
Let $(F,\E),(G,\E) \in \SSE$ be two soft sets over a common universe $\U$,
the \df{soft difference} of $(F,\E)$ and $(G,\E)$,
denoted by $(F,\E) \softsetminus (G,\E)$, is the soft set $\left( F \setminus G, \E \right)$
where $F \setminus G : \E \to \PP(\U)$ is the set-valued mapping
defined by $(F \setminus G)(e) = F(e) \setminus G(e)$, for every $e \in \E$.
\end{definition}

Clearly, for every soft set $(F,\E) \in \SSE$, it results
$(F,\E)^\complement \, \softequal \, \absolutesoftset \softsetminus (F,\E)$.

%-- definition of soft union
\begin{definition}{\rm\cite{nazmul}}
\label{def:softunion}
Let $(F,\E), (G,\E) \in \SSE$ be two soft sets over a universe $\U$,
the \df{soft union} of $(F,\E)$ and $(G,\E)$, denoted by $(F,\E) \softcup (G,\E)$,
is the soft set $\left( F \cup G, \E \right)$
where $F \cup G: \E \to \PP(\U)$ is the set-valued mapping
defined by $(F \cup G)(e) = F(e) \cup G(e)$, for every $e \in \E$.
\end{definition}

%-- definition of soft intersection
\begin{definition}{\rm\cite{nazmul}}
\label{def:softintersection}
Let $(F,\E), (G,\E) \in \SSE$ be two soft sets over a universe $\U$,
the \df{soft intersection} of $(F,\E)$ and $(G,\E)$, denoted by $(F,\E) \softcap (G,\E)$,
is the soft set $\left( F \cap G, \E \right)$
where $F \cap G: \E \to \PP(\U)$ is the set-valued mapping
defined by $(F \cap G)(e) = F(e) \cap G(e)$, for every $e \in \E$.
\end{definition}

\begin{proposition}{\rm\cite{cagman2014}}
\label{pro:propertiesunionandintersection}
For every soft set $(F,\E) \in \SSE$, the following hold:
\begin{enumerate}[label={\rm(\arabic*)}]
\item $(F,\E) \, \softcup \, (F,\E) \, \softequal \, (F,\E) .$
\item $(F,\E) \, \softcup \, \nullsoftset \, \softequal \, (F,\E) .$
\item $(F,\E) \, \softcup \, \absolutesoftset \, \softequal \, \absolutesoftset .$
\item $(F,\E) \, \softcap \, (F,\E) \, \softequal \, (F,\E) .$
\item $(F,\E) \, \softcap \, \nullsoftset \, \softequal \, \nullsoftset .$
\item $(F,\E) \, \softcap \, \absolutesoftset \, \softequal \, (F,\E) .$
\end{enumerate}
\end{proposition}

%-- definition of soft disjunct
\begin{definition}{\rm\cite{al-khafaj}}
\label{def:softdisjunct}
Two soft sets $(F,\E)$ and $(G,\E)$ over a common universe $\U$
are said to be \df{soft disjoint} if their soft intersection is the soft null set,
i.e. if $(F,\E) \softcap (G,\E) \, \softequal \, \nullsoftset$.
If two soft sets are not soft disjoint, we also say that they \df{soft meet} each other.
In particular, if $(F,\E) \softcap (G,\E) \, \softnotequal \, \nullsoftset$
we say that $(F,\E)$ \df{soft meets} $(G,\E)$.
\end{definition}

%-- relationship between soft difference and soft complement
\begin{proposition}{\rm\cite{shabir}}
\label{pro:softdifferenceandsoftcomplement}
Let $(F,\E), (G,\E) \in \SSE$ be two soft sets over a universe $\U$, we have that
$(F,\E) \softsetminus (G,\E) \softequal (F,\E) \softcap (G,\E)^\complement$.
\end{proposition}

The notions of soft union and intersection admit some obvious generalizations to a family
with any number of soft sets.

%-- definition of generalized soft union
\begin{definition}{\rm\cite{nazmul}}%{\rm\cite{ali}}
\label{def:generalizedsoftunion}
Let $\left\{(F_i,\E) \right\}_{i\in I} \subseteq \SSE$ be a nonempty subfamily
of soft sets over a universe $\U$,
the (generalized) \df{soft union} of $\left\{(F_i,\E) \right\}_{i\in I}$,
denoted by $\softbigcup_{i \in I} (F_i,\E) $,
is defined by $\left(\bigcup_{i \in I} F_i, \E \right)$
where $\bigcup_{i \in I} F_i: \E \to \PP(\U)$ is the set-valued mapping
defined by $\left(\bigcup_{i \in I} F_i\right)(e) = \bigcup_{i \in I} F_i(e)$, for every $e \in \E$.
\end{definition}

%-- definition of generalized soft intersection
\begin{definition}{\rm\cite{nazmul}}
\label{def:generalizedsoftintersection}
Let $\left\{(F_i,\E) \right\}_{i\in I} \subseteq \SSE$ be a nonempty subfamily
of soft sets over a universe $\U$,
the (generalized) \df{soft intersection} of $\left\{(F_i,\E) \right\}_{i\in I}$,
denoted by $\softbigcap_{i \in I} (F_i,\E) $,
is defined by $\left(\bigcap_{i \in I} F_i, \E \right)$
where $\bigcap_{i \in I} F_i: \E \to \PP(\U)$ is the set-valued mapping
defined by $\left(\bigcap_{i \in I} F_i\right)(e) = \bigcap_{i \in I} F_i(e)$, for every $e \in \E$.
\end{definition}

\begin{proposition}
\label{pro:generalizeddemorganlaws}
Let $\left\{(F_i,\E) \right\}_{i\in I} \subseteq \SSE$ be a nonempty subfamily
of soft sets over a universe $\U$, it results:
\begin{enumerate}[label={\rm(\arabic*)}]%[label=(\roman*)]
\item $\left( \softbigcup_{i \in I} (F_i,\E) \right)^{\!\complement} \softequal \;
\softbigcap_{i \in I} (F_i,\E)^\complement \,.$
\item $\left( \softbigcap_{i \in I} (F_i,\E) \right)^{\!\complement} \softequal \;
\softbigcup_{i \in I} (F_i,\E)^\complement \,.$
\end{enumerate}
\end{proposition}

%-----------------------------------------

%-- definition of soft point
\begin{definition}{\rm\cite{xie}}
\label{def:softpoint}
A soft set $(F,\E) \in \SSE$ over a universe $\U$ is said to be a \df{soft point} over $U$
if it has only one non-empty approximation which is a singleton,
i.e. if there exists some parameter $\alpha \in \E$
and an element $p \in \U$ such that
$F(\alpha) = \{ p \}$ and $F(e)=\emptyset$ for every $e \in \E \setminus \{ \alpha \}$.
Such a soft point is usually denoted by $(p_\alpha, \E)$.
The singleton $\{ p \}$ is called the \textit{support set} of the soft point
and $\alpha$ is called the \textit{expressive parameter} of $(p_\alpha, \E)$.
\end{definition}

\begin{remark}
\label{rem:softpoint}
In other words, a soft point $(p_\alpha, \E)$ is a soft set corresponding to
the set-valued mapping $p_\alpha : \E \to \mathbb(U)$ that, for any $e \in \E$, is defined by
$$ p_\alpha(e) =
\left\{
\begin{array}{ll}
\{ p \} & \mbox{ if } e = \alpha \\
\emptyset & \mbox{ if } e \in \E \setminus \{\alpha \} \\
\end{array}
\right. .
$$
\end{remark}

\begin{definition}{\rm\cite{xie}}
\label{def:setofsoftpoints}
The set of all the soft points over a universe $\U$ with respect to a set of parameters $\E$
will be denoted by $\SPE$.
\end{definition}

Since any soft point is a particular soft set, it is evident that $\SPE \subseteq \SSE$.

\begin{definition}{\rm\cite{xie}}
\label{def:softpointsoftbelongstosoftset}
Let $(p_\alpha, \E) \in \SPE$ and $(F,\E) \in \SSE$
be a soft point and a soft set over a common universe $\U$, respectively.
We say that \df{the soft point $(p_\alpha, \E)$ soft belongs to the soft set $(F,\E)$}
and we write $(p_\alpha, \E) \softin (F,\E)$, if the soft point is a soft subset of the soft set,
i.e. if $(p_\alpha, \E) \softsubseteq (F,\E)$
and hence if $p \in F(\alpha)$.
%-----
We also say that \df{the soft point $(p_\alpha, \E)$ does not belongs to the soft set $(F,\E)$}
and we write $(p_\alpha, \E) \softnotin (F,\E)$, if the soft point is not a soft subset of the soft set,
i.e. if $(p_\alpha, \E) \softnotsubseteq (F,\E)$
and hence if $p \notin F(\alpha)$.
\end{definition}

\begin{definition}{\rm\cite{das}}
\label{def:equalitysoftpoints}
Let $(p_\alpha, \E), (q_\beta, \E) \in \SPE$ be two soft points over a common universe $\U$,
we say that $(p_\alpha, \E)$ and $(q_\beta, \E)$ are \df{soft equal},
and we write $(p_\alpha, \E) \softequal (q_\beta, \E)$,
if they are equals as soft sets and hence if $p = q$ and $\alpha = \beta$.
\end{definition}

\begin{definition}{\rm\cite{das}}
\label{def:distinctssoftpoints}
We say that two soft points $(p_\alpha, \E)$ and $(q_\beta, \E)$ are \df{soft distincts},
and we write $(p_\alpha, \E) \softnotequal (q_\beta, \E)$,
if and only if $p\ne q$ or $\alpha \ne \beta$.
\end{definition}

The notion of soft point allows us to express the soft inclusion in a more familiar way.

%-- a soft set is soft contained in another one if and only if any soft point
%   soft belonging to the first soft sets, soft belongs to the second too
\begin{proposition}%{\rm\cite{nordo2019sub}}
\label{pro:softinclusionbysoftpoint}
Let $(F,\E), (G,\E) \in \SSE$ be two soft sets over a common universe $\U$ respect
to a parameter set $\E$, then $(F,\E) \softsubseteq (G,\E)$
if and only if for every soft point $(p_\alpha, \E) \softin (F,\E)$
it follows that $(p_\alpha, \E) \softin (G,\E)$.
\end{proposition}
\begin{proof}
Suppose that $(F,\E) \softsubseteq (G,\E)$. Then, for every $(p_\alpha, \E) \softin (F,\E)$,
by Definition \ref{def:softpointsoftbelongstosoftset}, we have that $p \in F(\alpha)$.
Since, by Definition \ref{def:softsubset}, we have in particular that $F(\alpha) \subseteq G(\alpha)$,
it follows that $p \in G(\alpha)$, which,
by Definition \ref{def:softpointsoftbelongstosoftset}, is equivalent to say that $(p_\alpha, \E) \softin (G,\E)$.
\\
Conversely, suppose that for every soft point $(p_\alpha, \E) \softin (F,\E)$
it follows $(p_\alpha, \E) \softin (G,\E)$.
Then, for every $e \in \E$ and any $p \in F(e)$, by Definition \ref{def:softpointsoftbelongstosoftset},
we have that the soft point $(p_e,\E) \softin (F,\E)$. So, by our hypotesis, it follows that
$(p_e,\E) \softin (G,\E)$ which is equivalent to $p \in G(e)$.
This proves that $F(e)\subseteq G(e)$ for every $e \in \E$ and so that $(F,\E) \softsubseteq (G,\E)$.
\end{proof}

%-------------------------------------------------------------

%-- definition of sub soft set
\begin{definition}{\rm\cite{hussain}}
\label{def:subsoftset}
Let $(F,\E) \in \SSE$ be a soft set over a universe $\U$
and $V$ be a nonempty subset of $\U$,
the \df{sub soft set} of $(F,\E)$ over $V$, is the soft set $\left(\rl[V] F,\E \right)$,
where $\rl[V] F: \E \to \PP(\U)$ is the set-valued mapping
defined by $\rl[V] F(e) = F(e) \cap V$, for every $e \in \E$.
\end{definition}

\begin{remark}
\label{rem:subsoftset}
Using Definitions \ref{def:constantsoftset} and \ref{def:softintersection},
it is a trivial matter to verify that
a sub soft set of $(F,\E)$ over $V$ can also be expressed as
$\left(\rl[V] F,\E \right) \softequal \, (F,\E) \softcap (\tilde{V}, \E)$.
\\
Furthermore, it is evident that the sub soft set $\left(\rl[V] F,\E \right)$ above defined
belongs to the set of all the soft sets over $V$ with respect to the set of parameters $\E$,
which is contained in the set of all the soft sets over the universe $\U$ with respect to $\E$,
that is $\left(\rl[V] F,\E \right) \in \SSE[V] \subseteq \SSE$.
\end{remark}

%-----------------------------------------------

%---- definition of soft product of soft sets
\begin{definition}{\rm\cite{babitha2010,kazanci}}
\label{def:softproductofsoftsets}
Let $\left\{(F_i,\E_i) \right\}_{i\in I}$ be a family of soft sets
over a universe set $\U_i$ with respect to a set of parameters $\E_i$ (with $i \in I$), respectively.
Then the \df{soft product} (or, more precisely, the \df{soft cartesian product})
of $\left\{(F_i,\E_i) \right\}_{i\in I}$,
denoted by $\softprod_{i \in I} (F_i,\E_i)$,
is the soft set $\left( \prod_{i \in I} F_i , \prod_{i \in I} \E_i \right)$
over the (usual) cartesian product $\prod_{i \in I} \U_i$
and with respect to the set of parameters $\prod_{i \in I} \E_i$,
where $\prod_{i \in I} F_i : \prod_{i \in I} \E_i \to \PP\left( \prod_{i \in I} \U_i \right)$
is the set-valued mapping defined by $\prod_{i \in I} F_i \left( \langle e_i \rangle_{i \in I} \right)
= \prod_{i \in I} F_i(e_i)$, for every $\langle e_i \rangle_{i \in I} \in \prod_{i \in I} \E_i$.
\end{definition}

%-- a soft point belongs to a soft cartesian product if and only if
%   every its component soft belongs to the corresponding soft factor
\begin{proposition}{\rm\cite{nordo2019prod}}
\label{pro:softpointsoftbelongingtosoftproduct}
Let $\softprod_{i \in I} (F_i,\E_i)$ be the soft product of a family
$\left\{(F_i,\E_i) \right\}_{i\in I}$  of soft sets
over a universe set $\U_i$ with respect to a set of parameters $\E_i$ (with $i \in I$),
and let $\left( p_\alpha , \prod_{i \in I} \E_i \right) \in \SPG{\prod_{i \in I} \U_i}{\prod_{i \in I} \E_i}$
be a soft point of the product $\prod_{i \in I} \U_i$,
where $p = \langle p_i \rangle_{i \in I} \in \prod_{i \in I} \U_i$ and
$\alpha = \langle \alpha_i \rangle_{i \in I} \in \prod_{i \in I} \E_i$, then
$\left( p_\alpha , \prod_{i \in I} \E_i \right) \softin \, \softprod_{i \in I} (F_i,\E_i)$
if and only if
$\left( (p_i)_{\alpha_i}, \E_i \right) \softin (F_i, \E_i)$ for every $i \in I$.
\end{proposition}
%----
\begin{proof}
In fact, by using Definitions \ref{def:softproductofsoftsets} and \ref{def:softpointsoftbelongstosoftset},
$\left( p_\alpha , \prod_{i \in I} \E_i \right) \softin \softprod_{i \in I} (F_i,\E_i)$
means that
$p \in \left( \prod_{i \in I} F_i \right)(\alpha)$
that is $\langle p_i \rangle_{i \in I} \in \left( \prod_{i \in I} F_i \right)
\left( \langle \alpha_i \rangle_{i \in I} \right)$
which corresponds to say tat
$p_i \in F_i(\alpha_i)$ for every $i \in I$
which, by Definition \ref{def:softpointsoftbelongstosoftset}, is equivalent to
$\left( (p_i)_{\alpha_i}, \E_i \right) \softin (F_i, \E_i)$ for every $i \in I$.
\end{proof}

%-- soft nullity of soft product
\begin{corollary}{\rm\cite{nordo2019prod}}
\label{cor:softnullityofsoftproduct}
The soft product of a family $\left\{(F_i,\E_i) \right\}_{i\in I}$
of soft sets over a universe set $\U_i$
with respect to a set of parameters $\E_i$ (with $i \in I$)
is null if and only if at least one of its soft sets is null, that is
$\softprod_{i \in I} (F_i,\E_i) \, \softequal \, \nullsoftsetG{\prod_{i \in I} \E_i}$
if$f$ there exists some $j \in I$ such that $(F_j,\E_j) \softequal \nullsoftset$.
%$$\softprod_{i \in I} (F_i,\E_i) \, \softequal \, \nullsoftsetG{\prod_{i \in I} \E_i}
%\quad \Longleftrightarrow \quad
%\exists j \in I : \; (F_j,\E_j) \softequal \nullsoftset .$$
\end{corollary}

%-- monotonic of soft product
\begin{proposition}{\rm\cite{nordo2019prod}}
\label{pro:monotonicofsoftproduct}
Let $\left\{(F_i,\E_i) \right\}_{i\in I}$ and $\left\{(G_i,\E_i) \right\}_{i\in I}$
be two families of soft sets over a universe set $\U_i$
with respect to a set of parameters $\E_i$ (with $i \in I$), such that
$(F_i,\E_i) \softsubseteq (G_i,\E_i)$ for every $i \in I$, then
their respective soft products are such that
$\softprod_{i \in I} (F_i,\E_i) \, \softsubseteq \, \softprod_{i \in I} (G_i,\E_i)$.
\end{proposition}

%-- distributive property of soft product respect to soft intersection
\begin{proposition}{\rm\cite{kazanci}}
\label{pro:distributivepropertyofsoftproductrespecttosoftintersection}
Let $\left\{(F_i,\E_i) \right\}_{i\in I}$ and $\left\{(G_i,\E_i) \right\}_{i\in I}$
be two families of soft sets over a universe set $\U_i$
with respect to a set of parameters $\E_i$ (with $i \in I$), then it results:
$$\softprod_{i \in I} \left( (F_i,\E_i) \softcap (G_i, \E_i) \right)
\, \softequal \, \softprod_{i \in I} (F_i,\E_i) \, \softcap \, \softprod_{i \in I} (G_i,\E_i) \, . $$
\end{proposition}

%--------------------------------------------------------------------------------------------

%-- our definition is an adaption of Kharal-Ahmad's one
%   since our soft sets are defined on the whole set of parameters
According to Remark \ref{rem:sameparameter} the following notions by
Kharal and Ahmad have been simplified and slightly modified for soft sets
defined on a common parameter set.

%-- definition of soft mapping
\begin{definition}{\rm\cite{kharal}}
\label{def:softmapping}
Let $\SSG{\U}{\E}$ and $\SSG{\U'}{\E'}$ be two sets of soft open sets
over the universe sets $\U$ and $\U'$ with respect to the sets of parameters $\E$ and $\E'$, respectively.
and consider a mapping $\varphi: \U \to \U'$ between the two universe sets and
a mapping $\psi: \E \to \E'$ between the two set of parameters.
The mapping $\varphi_\psi : \SSG{\U}{\E} \to \SSG{\U'}{\E'}$
which maps every soft set $(F,\E)$ of $\SSG{\U}{\E}$
to a soft set $(\varphi_\psi(F), \E')$ of $\SSG{\U'}{\E'}$,
denoted by $\varphi_\psi (F,\E)$,
where $\varphi_\psi(F) : \E' \to \PP(\U')$ is the set-valued mapping defined by
$\varphi_\psi(F)(e') = \bigcup \left\{ \varphi(F(e)) : e\in \psi^{-1}(\{e'\}) \right\}$
for every $e' \in \E'$,
is called a \df{soft mapping} from $\U$ to $\U'$ induced by the mappings $\varphi$ and $\psi$,
while the soft set $\varphi_\psi (F,\E) \softequal (\varphi_\psi(F), \E')$
is said to be the \df{soft image} of the soft set $(F,\E)$
under the soft mapping $\varphi_\psi : \SSG{\U}{\E} \to \SSG{\U'}{\E'}$.
\\
The soft mapping $\varphi_\psi : \SSG{\U}{\E} \to \SSG{\U'}{\E'}$ is said
\df{injective} (respectively \df{surjective}, \df{bijective})
if the mappings $\varphi: \U \to \U'$ and $\psi: \E \to \E'$ are both
injective (resp. surjective, bijective).
\end{definition}

\begin{remark}
\label{rem:softmapping}
In other words a soft mapping $\varphi_\psi : \SSG{\U}{\E} \to \SSG{\U'}{\E'}$
matches every set-valued mapping $F: \E \to \PP(\U)$
to a set-valued mapping $\varphi_\psi(F) : \E' \to \PP(\U')$
which, for every $e' \in \E'$, is defined by
$$
\varphi_\psi(F)(e') = \left\{
\begin{array}{ll}
\bigcup_{e\in \psi^{-1}(\{e'\})} \varphi(F(e))
& \text{ if } \psi^{-1}(\{e'\}) \ne \emptyset
\\[2mm]
\emptyset & \text{ otherwise}
\end{array}
\right. .
$$
In particular, if the soft mapping $\varphi_\psi$ is bijective,
the set-valued mapping $\varphi_\psi(F) : \E' \to \PP(\U')$
is defined simply by
$\varphi_\psi(F)(e') = \varphi \left( F \left(  \psi^{-1}(e') \right) \right)$,
for every $e' \in \E'$.
\\
Let us also note that in some paper (see, for example, \cite{varol2013}) the soft mapping
$\varphi_\psi$ is denoted with  $(\varphi, \psi) : \SSG{\U}{\E} \to \SSG{\U'}{\E'}$.
\end{remark}

It is worth noting that soft mappings between soft sets behaves similarly
to usual (crisp) mappings in the sense that they maps soft points to soft points,
as proved in the following property.

%----- soft image of a soft point
\begin{proposition}
\label{pro:softimageofasoftpoint}
Let $\varphi_\psi : \SSG{\U}{\E} \to \SSG{\U'}{\E'}$ be a soft mapping
induced by the mappings $\varphi: \U \to \U'$ and $\psi: \E \to \E'$
between the two sets $\SSG{\U}{\E}, \SSG{\U'}{\E'}$ of soft sets.
and consider a soft point $(p_\alpha,\E)$ of $\SPE$.
Then the soft image $\varphi_\psi (p_\alpha,\E)$ of the soft point $(p_\alpha,\E)$
under the soft mapping $\varphi_\psi$
is the soft point $\left( {\varphi(p)}_{\psi(\alpha)}, \E'\right)$,
i.e. $\varphi_\psi (p_\alpha,\E) \softequal \left( {\varphi(p)}_{\psi(\alpha)}, \E'\right)$.
\end{proposition}
%-----
\begin{proof}
Let $(p_\alpha,\E)$ be a soft point of $\SPE$, by Definition \ref{def:softmapping},
its soft image $\varphi_\psi (p_\alpha,\E) $
under the soft mapping $\varphi_\psi : \SSG{\U}{\E} \to \SSG{\U'}{\E'}$
is the soft set $\left( \varphi_\psi \left( p_\alpha \right) , \E' \right)$
corresponding to the set-valued mapping
$\varphi_\psi \left( p_\alpha \right) : \E' \to \PP(\U')$
which, for every $e' \in \E'$, is defined by
$\varphi_\psi(p_\alpha)(e') = \bigcup \left\{ \varphi(p_\alpha(e))
: e\in \psi^{-1}(\{e'\}) \right\}$.
%-----
Now, if $e' = \psi(\alpha)$ we have that:
\begin{align*}
\varphi_\psi(p_\alpha) \left( \psi(\alpha) \right) &= \bigcup \left\{ \varphi(p_\alpha(e))
  : e\in \psi^{-1}(\{ \psi(\alpha) \}) \right\}
\\
&= \bigcup \left\{ \varphi\left(p_\alpha(e)\right) : e\in \E , \, \psi(e)=\psi(\alpha) \right\}
\\
&= \bigcup \left\{ \varphi\left(p_\alpha(e)\right) : e = \alpha  \right\}
  \cup
  \bigcup \left\{ \varphi\left(p_\alpha(e)\right) :
  e\in \E\setminus \{ \alpha \} , \psi(e)=\psi(\alpha) \right\}
\\
&= \left\{ \varphi \left( p_\alpha(\alpha) \right) \right\}
  \cup
  \bigcup \left\{ \varphi\left(p_\alpha(e)\right)= \emptyset :
  e\in \E\setminus \{ \alpha \} , \, \psi(e)=e' \right\}
\\
&= \left\{ \varphi (p)  \right\} \cup \emptyset
\\
&= \left\{ \varphi (p)  \right\}
\end{align*}
%-----
while, for every $e' \in \E' \setminus \left\{ \psi(\alpha) \right\}$,
we have that $\psi(\alpha) \ne e'$ and so it follows that:
\begin{align*}
\varphi_\psi(p_\alpha)(e') &= \bigcup \left\{ \varphi(p_\alpha(e))
: e\in \psi^{-1}(\{e'\}) \right\}
\\
&= \bigcup \left\{ \varphi\left(p_\alpha(e)\right) : e\in \E , \, \psi(e)=e' \right\}
\\
&= \bigcup \left\{ \varphi\left(p_\alpha(e)\right) : e\in \E \setminus \{\alpha\}, \, \psi(e)=e' \right\}
\\
&= \bigcup \left\{ \varphi\left(p_\alpha(e)\right)=\emptyset : e\in \E \setminus \{\alpha\}, \, \psi(e)=e' \right\}
\\
&= \emptyset \, .
\end{align*}
This proves that the set-valued mapping $\varphi_\psi \left( p_\alpha \right) : \E' \to \PP(\U')$
sends the parameter $\psi(\alpha)$ to the singleton $ \left\{ \varphi (p)  \right\} $
and maps every other parameters of $\E' \setminus \left\{ \psi(e) \right\}$ to the empty set,
and so, by Definition \ref{def:softpoint}, this means that
the soft image $\varphi_\psi (p_\alpha,\E)$ of the soft point $(p_\alpha,\E) \in \SPE$
under the soft mapping $\varphi_\psi : \SSG{\U}{\E} \to \SSG{\U'}{\E'}$
is the soft point in $\SPG{\U'}{\E'}$
having $\left\{ \varphi(p) \right\}$ as support set
and $\psi(\alpha)$ as expressive parameter,
that is $\left( {\varphi(p)}_{\psi(\alpha)}, \E'\right)$.
\end{proof}

\begin{corollary}
\label{cor:injectivityofasoftmapping}
Let $\varphi_\psi : \SSG{\U}{\E} \to \SSG{\U'}{\E'}$ be a soft mapping
induced by the mappings $\varphi: \U \to \U'$ and $\psi: \E \to \E'$
between the two sets $\SSG{\U}{\E}, \SSG{\U'}{\E'}$ of soft sets,
then $\varphi_\psi$ is injective if and only if
its soft images of every distinct pair of soft points are distinct too, i.e. if
for every $(p_\alpha,\E), (q_\beta,\E) \in \SPE$ such that $(p_\alpha,\E) \softnotequal (q_\beta,\E)$
it follows that $\varphi_\psi (p_\alpha,\E) \softnotequal \varphi_\psi (q_\beta,\E)$.
\end{corollary}
\begin{proof}
It easily derives from Definitions \ref{def:distinctssoftpoints} and \ref{def:softmapping},
and Proposition \ref{pro:softimageofasoftpoint}.
\end{proof}

%----------------------

%-- definition of soft inverse image
\begin{definition}{\rm\cite{kharal}}
\label{def:softinverseimage}
Let $\varphi_\psi : \SSG{\U}{\E} \to \SSG{\U'}{\E'}$ be a soft mapping
induced by the mappings $\varphi: \U \to \U'$ and $\psi: \E \to \E'$
between the two sets $\SSG{\U}{\E}, \SSG{\U'}{\E'}$ of soft sets
and consider a soft set $(G,\E')$ of $\SSG{\U'}{\E'}$.
The \df{soft inverse image} of $(G,\E')$ under the soft mapping
$\varphi_\psi : \SSG{\U}{\E} \to \SSG{\U'}{\E'}$,
denoted by $\varphi_\psi^{-1} (G,\E')$ is the
soft set $(\varphi_\psi^{-1}(G), \E')$ of $\SSG{\U}{\E}$
where $\varphi_\psi^{-1}(G) : \E \to \PP(\U)$ is the set-valued mapping
defined by $\varphi_\psi^{-1}(G)(e) = \varphi^{-1}\left( G\left( \psi(e) \right) \right)$
for every $e \in \E$.
\end{definition}

%-- properties of soft mappings
\begin{proposition}{\rm\cite{aygunoglu,kharal,zorlutuna}}
\label{pro:propertiesofsoftmappings}
Let $\varphi_\psi : \SSG{\U}{\E} \to \SSG{\U'}{\E'}$ be a soft mapping
induced by the mappings $\varphi: \U \to \U'$ and $\psi: \E \to \E'$
and let $(F,\E), (F_i, \E) \in \SSG{\U}{\E}$ and $(G,\E'),(G_i, \E') \in \SSG{\U'}{\E'}$
be soft sets over $\U$ and $\U'$, respectively, then the following hold:
\begin{enumerate}[label={\rm(\arabic*)}]
\item $\varphi_\psi\!\nullsoftsetG{\E}  \softequal  \nullsoftsetG{\E'} \!.$
\item $\varphi_\psi^{-1}\!\nullsoftsetG{\E'} \softequal \nullsoftsetG{\E} \!.$
\item $\varphi_\psi^{-1}\!\absolutesoftsetG{\U'}{\E'}  \softequal  \absolutesoftsetG{\U}{\E} \!.$
\item $(F,\E) \, \softsubseteq \, \varphi_\psi^{-1} \!\left( \varphi_\psi (F,\E) \right)$
  and the soft equality holds when $\varphi_\psi$ is injective.
\item $\varphi_\psi \! \left( \varphi_\psi^{-1} (G,\E') \right) \softsubseteq \, (G,\E')$
  and the soft equality holds when $\varphi_\psi$ is surjective.
\item $\varphi_\psi^{-1} \! \left( (G,\E')^\complement \right) \softequal
  \left( \varphi_\psi^{-1}(G,\E') \right)^{\!\complement} \!.$
\item if $(F_1, \E) \, \softsubseteq \, (F_2,\E) .$
  then $\varphi_\psi (F_1,\E) \, \softsubseteq \, \varphi_\psi (F_2,\E) .$
\item if $(G_1, \E') \, \softsubseteq \, (G_2,\E') .$
  then $\varphi_\psi^{-1} (G_1,\E') \, \softsubseteq \, \varphi_\psi^{-1} (G_2,\E') .$
\item $\varphi_\psi \! \left( \softbigcup_{i \in I} (F_i,\E) \right)
  \softequal \,\, \softbigcup_{i \in I} \varphi_\psi(F_i,\E) .$
\item $\softbigcap_{i \in I} \varphi_\psi(F_i,\E) \; \softsubseteq \;
 \varphi_\psi \! \left( \softbigcap_{i \in I} (F_i,\E) \right) \!.$
\item $\varphi_\psi^{-1} \! \left( \softbigcup_{i \in I} (G_i,\E') \right)
  \softequal \; \softbigcup_{i \in I} \varphi_\psi^{-1}(G_i,\E') .$
\item $\varphi_\psi^{-1} \! \left( \softbigcap_{i \in I} (G_i,\E') \right)
  \softequal \; \softbigcap_{i \in I} \varphi_\psi^{-1}(G_i,\E') .$
\end{enumerate}
\end{proposition}

%-- monotonic soft images and soft inverse images
\begin{proposition}{\rm\cite{kharal}}
\label{pro:monotonicsoftimagesandsoftinverseimages}
Let $\varphi_\psi : \SSG{\U}{\E} \to \SSG{\U'}{\E'}$ be a soft mapping
induced by the mappings $\varphi: \U \to \U'$ and $\psi: \E \to \E'$
and let $(F,\E), (G, \E) \in \SSG{\U}{\E}$  and $(F',\E'),(G', \E') \in \SSG{\U'}{\E'}$
be soft sets over $\U$ and $\U'$, respectively, then the following hold:
\begin{enumerate}[label={\rm(\arabic*)}]
\item $(F,\E) \, \softsubseteq \, (G,\E)$ implies
  $\varphi_\psi (F,\E) \, \softsubseteq \, \varphi_\psi (G,\E) .$
\item $(F',\E) \, \softsubseteq \, (G',\E')$ implies
  $\varphi_\psi^{-1}(F',\E') \, \softsubseteq \, \varphi_\psi^{-1}(G',\E') .$
\end{enumerate}
\end{proposition}

%-- monotonic soft images and soft inverse images
\begin{corollary}
\label{cor:monotonicsoftimagesandsoftinverseimages}
Let $\varphi_\psi : \SSG{\U}{\E} \to \SSG{\U'}{\E'}$ be a soft mapping
induced by the mappings $\varphi: \U \to \U'$ and $\psi: \E \to \E'$.
If $(F,\E) \in \SSG{\U}{\E}$ and $(F',\E') \in \SSG{\U'}{\E'}$
are soft sets over $\U$ and $\U'$, respectively
and $(p_\alpha, \E) \in \SPG{\U}{\E}$ and $(q_\beta, \E') \in \SPG{\U'}{\E'}$
are soft points over $\U$ and $\U'$, respectively, then the following hold:
\begin{enumerate}[label={\rm(\arabic*)}]
\item $(p_\alpha, \E) \softin (F,\E)$ implies
  $\varphi_\psi (p_\alpha, \E) \softin \varphi_\psi (F,\E) .$
\item $(q_\beta, \E') \softin (F',\E')$ implies
  $\varphi_\psi^{-1} (q_\beta, \E') \, \softsubseteq \, \varphi_\psi^{-1}(F',\E') .$
\end{enumerate}
\end{corollary}

%---- definition of soft inverse mapping
\begin{definition}%{\rm\cite{aygunoglu}}
\label{def:softinversemapping}
Let $\varphi_\psi : \SSG{\U}{\E} \to \SSG{\U'}{\E'}$ be a bijective soft mapping
induced by the mappings $\varphi: \U \to \U'$ and $\psi: \E \to \E'$.
The \df{soft inverse mapping} of $\varphi_\psi$,
denoted by $\varphi_\psi^{-1}$,
is the soft mapping $\varphi_\psi^{-1} = \left( \varphi^{-1} \right)_{\psi^{-1}} : \SSG{\U'}{\E'} \to \SSG{\U}{\E}$
induced by the inverse mappings $\varphi^{-1}: \U' \to \U$ and $\psi^{-1}: \E' \to \E$
of the mappings $\varphi$ and $\psi$, respectively.
\end{definition}

%-- note: the image of a soft inverse mapping coincides with the soft inverse image of the soft mapping
\begin{remark}
\label{rem:imageosoftinversemapping}
Evidently, the soft inverse mapping $\varphi_\psi^{-1} : \SSG{\U'}{\E'} \to \SSG{\U}{\E}$
of a bijective soft mapping $\varphi_\psi : \SSG{\U}{\E} \to \SSG{\U'}{\E'}$ is also bijective
and its soft image of a soft set in $\SSG{\U'}{\E'}$ coincides with the soft inverse image of
the corresponding soft set under the soft mapping $\varphi_\psi$.
\end{remark}

%-- composition of soft mappings
\begin{definition}{\rm\cite{aygunoglu}}
\label{def:softcompositionofsoftmappings}
Let $\SSG{\U}{\E}, \SSG{\U'}{\E'}$ and $\SSG{\U''}{\E''}$ be three sets of soft open sets
over the universe sets $\U, \U', \U''$ with respect to the sets of parameters $\E, \E', \E''$, respectively,
and $\varphi_\psi : \SSG{\U}{\E} \to \SSG{\U'}{\E'}$, $\gamma_\delta : \SSG{\U}{\E'} \to \SSG{\U'}{\E''}$
be two soft mappings between such sets,
then the \df{soft composition} of the soft mappings $\varphi_\psi $ and $\gamma_\delta$,
denoted by $\gamma_\delta \softcirc \varphi_\psi$ is the
soft mapping $\left( \gamma \circ \varphi \right)_{\delta \circ \psi} : \SSG{\U}{\E} \to \SSG{\U''}{\E''}$
induced by the compositions $\gamma \circ \varphi : \U \to \U''$
of the mappings $\varphi$ and $\gamma$ between the universe sets
and $\delta \circ \psi : \E \to \E''$ of the mappings $\psi$ and $\delta$ between the parameter sets.
\end{definition}

%----------------------------------------------------------------------

%---- soft topological spaces

The notion of soft topological spaces as topological spaces defined over a initial universe
with a fixed set of parameters was introduced in 2011 by Shabir and Naz \cite{shabir}.

%---- definition of soft topology
\begin{definition}{\rm\cite{shabir}}
\label{def:softtopology}
Let $X$ be an initial universe set, $\E$ be a nonempty set of parameters with respect to $X$
and $\Tau \subseteq \SSE[X]$ be a family of soft sets over $X$, we say that
$\Tau$ is a \df{soft topology} on $X$ with respect to $\E$ if the following four conditions are satisfied:
\begin{enumerate}[label={\rm(\roman*)}]
\item the null soft set belongs to $\Tau$, i.e. $\nullsoftset \in \Tau .$
\item the absolute soft set belongs to $\Tau$, i.e. $\absolutesoftset[X] \in \Tau .$
\item the soft intersection of any two soft sets of $\Tau$ belongs to $\Tau$, i.e.
for every $(F,\E), (G,\E) \in \Tau$ then $(F,\E) \softcap (G,\E) \in \Tau .$
\item the soft union of any subfamily of soft sets in $\Tau$ belongs to $\Tau$, i.e.
for every $\left\{(F_i,\E) \right\}_{i\in I} \subseteq \Tau$ then
$\softbigcup_{i \in I} (F_i,\E) \in \Tau .$
\end{enumerate}
The triplet $(X, \Tau, \E)$ is called a \df{soft topological space} (or soft space, for short)
over $X$ with respect to $\E$.
\\
In some case, when it is necessary to better specify the universal set and the set of parameters,
the topology will be denoted by $\Tau(X,\E)$.
\end{definition}

\begin{definition}{\rm\cite{shabir}}
\label{def:softopenset}
Let $(X,\Tau,\E)$ be a soft topological space over $X$ with respect to $\E$,
then the members of $\Tau$ are said to be \df{soft open set} in $X$.
\end{definition}

%-- comparison of soft topologies
%-- weaker (or coarser)
%-- stronger (or finer)
\begin{definition}{\rm\cite{hazra}}
\label{def:comparisonofsofttopologies}
Let $\Tau_1$ and $\Tau_2$ be two soft topologies over a common universe set $X$
with respect to a set of paramters $\E$.
We say that $\Tau_2$ is \df{finer} (or stronger) than $\Tau_1$
if $\Tau_1 \subseteq \Tau_2$ where $\subseteq$ is the usual set-theoretic relation of inclusion between crisp sets.
In the same situation, we also say that $\Tau_1$ is \df{coarser} (or weaker) than $\Tau_2$.
\end{definition}

%-- definition of soft closed set
\begin{definition}{\rm\cite{shabir}}
\label{def:softclosedset}
Let $(X,\Tau,\E)$ be a soft topological space over $X$ and $(F,\E)$ be a soft set over $X$.
We say that $(F,\E)$ is \df{soft closed set} in $X$ if its complement $(F,\E)^\complement$
is a soft open set, i.e. if $(F,\E)^\complement \in \Tau$.
\end{definition}

\begin{notation}
The family of all soft closed sets of a soft topological space $(X,\Tau,\E)$ over $X$ with respect to $\E$
will be denoted by $\sigma$,
or more precisely with $\sigma(X,\E)$ when it is necessary to specify
the universal set $X$ and the set of parameters $\E$.
\end{notation}

%---- properties of the family of soft closed sets
\begin{proposition}{\rm\cite{shabir}}
\label{pro:propertiesofsoftclosedsets}
Let $\sigma$ be the family of soft closed sets of a soft topological space $(X,\Tau,\E)$, the following hold:
\begin{enumerate}[label={\rm(\arabic*)}]
\item the null soft set is a soft closed set, i.e. $\nullsoftset \in \sigma .$
\item the absolute soft set is a soft closed set, i.e. $\absolutesoftset[X] \in \sigma .$
\item the soft union of any two soft closed sets is still a soft closed set, i.e.
for every $(C,\E), (D,\E) \in \sigma$ then $(C,\E) \softcup (D,\E) \in \sigma .$
\item the soft intersection of any subfamily of soft closed sets is still a soft closed set, i.e.
for every $\left\{(C_i,\E) \right\}_{i\in I} \subseteq \sigma$ then
$\softbigcap_{i \in I} (C_i,\E) \in \sigma .$
\end{enumerate}
\end{proposition}

%-------------------------------

%-- definition of soft base
\begin{definition}{\rm\cite{aygunoglu}}
\label{def:softbase}
Let $(X,\Tau,\E)$ be a soft topological space over $X$ and $\Bcal \subseteq \Tau$
be a non-empty subset of soft open sets.
We say that $\Bcal$ is a \df{soft open base} for $(X,\Tau,\E)$ if every soft open set of $\Tau$ can be
expressed as soft union of a subfamily of $\Bcal$, i.e. if for every $(F,\E) \in \Tau$
there exists some $\Acal \subset \Bcal$ such that
$(F,\E) = \softbigcup \left\{ (A,\E) : \, (A,\E) \in \Acal \right\}$.
\end{definition}

%-- characterization of soft open base
\begin{proposition}{\rm\cite{nazmul}}
\label{pro:characterizationofsoftopenbase}
Let $(X,\Tau,\E)$ be a soft topological space over $X$ and
$\Bcal \subseteq \Tau$ be a family of soft open sets of $X$.
Then $\Bcal$ is a soft open base for $(X,\Tau,\E)$
if and only if for every soft open set $(F,\E) \in \Tau$ and any soft point $(x_\alpha, \E) \softin (F,\E)$
there exists some soft open set $(B,\E) \in \Bcal$
such that $(x_\alpha, \E) \softin (B,\E) \softsubseteq (F,\E)$.
\end{proposition}

%-------------------------

%-- definition of soft neighbourhood
\begin{definition}{\rm\cite{zorlutuna}}
\label{pro:softneighbourhoodofasoftpoint}
Let $(X, \Tau, \E)$ be a soft topological space, $(N,\E) \in \SSE[X]$ be a soft set
and $(x_\alpha, \E) \in \SPE[X]$ be a soft point over a common universe $X$.
We say that $(N,\E)$ is a \df{soft neighbourhood} of the soft point $(x_\alpha, \E)$
if there is some soft open set soft containing the soft point and soft contained in the soft set,
that is if there exists some soft open set $(A,\E) \in \Tau$
such that $(x_\alpha, \E) \softin (A,\E) \softsubseteq (N,\E)$.
\end{definition}

%-- notation for the family of soft neighbourhoods
\begin{notation}
\label{not:familyofsoftneighbourhoodofasoftpoint}
The family of all soft neighbourhoods
(sometimes also called soft neighbourhoods system) of a soft point $(x_\alpha, \E) \in \SPE[X]$
in a soft topological space $(X, \Tau, \E)$ will be denoted by
$\Ncal_{(x_\alpha, \E)}$
(or more precisely with $\Ncal^\Tau_{(x_\alpha, \E)}$ if it is necessary to specify the topology).
\end{notation}

%-- definition of soft closure
\begin{definition}{\rm\cite{shabir}}
\label{def:softclosure}
Let $(X,\Tau,\E)$ be a soft topological space over $X$ and $(F,\E)$ be a soft set over $X$.
Then the \df{soft closure} of the soft set $(F,\E)$ with respect to the soft space $(X,\Tau,\E)$,
denoted by $\softcl{(F,\E)}$,
is the soft intersection of all soft closed set over $X$ soft containing $(F,\E)$, that is
$$\softcl{(F,\E)} \softequal \, \softbigcap \left\{ (C,\E) \in \sigma(X,\E) :
\, (F,\E) \softsubseteq (C,\E) \right\} .$$
\end{definition}

%---- properties of soft closure
\begin{proposition}{\rm\cite{shabir}}
\label{pro:propertiesofsoftclosure}
Let $(X,\Tau,\E)$ be a soft topological space over $X$, and $(F,\E)$ be a soft set over $X$.
Then the following hold:
\begin{enumerate}[label={\rm(\arabic*)}]
\item $\softcl{ \nullsoftset } \softequal \, \nullsoftset .$
\item $\softcl{ \absolutesoftset[X] } \softequal \, \absolutesoftset[X] .$
\item $(F,\E) \, \softsubseteq \, \softcl{(F,\E)} .$
\item $(F,\E)$ is a soft closed set over $X$ if and only if $\softcl{(F,\E)} \softequal \, (F,\E) .$
\item $\softclpar{ \softcl{(F,\E)} } \softequal \, \softcl{(F,\E)} .$
\end{enumerate}
\end{proposition}

\begin{proposition}{\rm\cite{shabir}}
\label{pro:sofsoftclosureofoperators}
Let $(X,\Tau,\E)$ be a soft topological space and
$(F,\E), (G,\E) \in \SSE[X]$ be two soft sets over a common universe $X$.
Then the following hold:
\begin{enumerate}[label={\rm(\arabic*)}]
\item $(F,\E) \, \softsubseteq \, (G,\E)$ implies $\softcl{(F,\E)} \, \softsubseteq \, \softcl{(G,\E)} .$
\item $\softclpar{ (F,\E) \, \softcup \, (G,\E) } \softequal \, \softcl{(F,\E)} \, \softcup \, \softcl{(G,\E)} .$
\item $\softclpar{ (F,\E) \, \softcap \, (G,\E) } \softsubseteq \, \softcl{(F,\E)} \, \softcap \, \softcl{(G,\E)} .$
\end{enumerate}
\end{proposition}

%--------------

%-- definition of soft adherent point
\begin{definition}{\rm\cite{xie}}
\label{def:softadherentpoint}
Let $(X,\Tau,\E)$ be a soft topological space, $(F,\E) \in \SSE[X]$ and $(x_\alpha,\E) \in \SPE[X]$
be a soft set and a soft point over the common universe $X$
with respect to the sets of parameters $\E$, respectively.
We say that $(x_\alpha,\E)$ is a \df{soft adherent point} (sometimes also called \df{soft closure point})
of $(F,\E)$ if it soft meets every soft neighbourhood of the soft point, that is if
for every $(N,\E) \in \Ncal_{(x_\alpha, \E)}$, $(F,\E) \, \softcap \, (N,\E) \, \softnotequal \, \nullsoftset$.
\end{definition}

As in the classical topological space, it is possible to prove that
the soft closure coincides with the set of all its soft adherent points.

%-- the soft closure of a soft set is the set of all its soft adherent points
\begin{proposition}{\rm\cite{xie}}
\label{pro:softclosureisthesetofsoftadherentpoints}
Let $(X,\Tau,\E)$ be a soft topological space, $(F,\E) \in \SSE[X]$ and $(x_\alpha,\E) \in \SPE[X]$
be a soft set and a soft point over the common universe $X$
with respect to the sets of parameters $\E$, respectively.
Then $(x_\alpha,\E) \softin \softcl{(F,\E)}$ if and only if
$(x_\alpha,\E)$ is a soft adherent point of $(F,\E)$.
\end{proposition}

%--------------------------

Having in mind the Definition \ref{def:subsoftset} we can recall the following proposition.

\begin{proposition}{\rm\cite{hussain}}
\label{pro:softrelativetopology}
Let $(X,\Tau,\E)$ be a soft topological space over $X$, and $Y$ be a nonempty subset of $X$, then
the family $\Tau_Y$ of all sub soft sets of $\Tau$ over $Y$, i.e.
$$\Tau_Y = \left\{ \left(\rl F,\E \right) : \, (F,\E) \in \Tau \right\}$$
is a soft topology on $Y$.
\end{proposition}

%-- definition of soft relative topology
\begin{definition}{\rm\cite{hussain}}
\label{def:softrelativetopology}
Let $(X,\Tau,\E)$ be a soft topological space over $X$, and let $Y$ be a nonempty subset of $X$,
the soft topology $\Tau_Y = \left\{ \left(\rl F,\E \right) : \, (F,\E) \in \Tau \right\}$
is said to be the \df{soft relative topology} of $\Tau$ on $Y$
and $(Y,\Tau_Y,\E)$ is called the \df{soft topological subspace} of $(X,\Tau,\E)$ on $Y$.
\end{definition}

%---- soft closed set in a subspace
\begin{proposition}
\label{pro:softclosedsetsinasoftsubspace}
Let $(X,\Tau,\E)$ be a soft topological space over $X$, and
$(Y,\Tau_Y,\E)$ be its soft topological subspace over the subset $Y\subseteq X$,
then a soft set $(D,\E) \in \SSE[Y]$ is a soft closed set respect to the soft subspace $(Y,\Tau_Y,\E)$
if and only if it is a sub soft set of some soft closed set of the soft space $(X,\Tau,\E)$,
i.e.
$$(D,\E) \in \sigma(Y,\E) \quad \Longleftrightarrow \quad
\exists (C,\E) \in \sigma(X,\E) : \, (D,\E) \,\, \softequal \, \left(\rl C,\E \right) .$$
\end{proposition}
\begin{proof}
It easily follows from Definitions \ref{def:softclosedset} and \ref{def:softrelativetopology},
Remark \ref{rem:subsoftset},
and Proposition \ref{pro:softdifferenceandsoftcomplement}.
\end{proof}

%---------------

%-- soft closure in a soft relative topology
\begin{proposition}{\rm\cite{nordo2019sub}}
\label{pro:softclosureinasoftrelativetopology}
Let $(X,\Tau,\E)$ be a soft topological space over $X$,
$(Y,\Tau_Y,\E)$ be its soft topological subspace on the subset $Y\subseteq X$, and
$(G,\E) \in \SSE[Y]$ be a soft set over $Y$ respect to the set of parameter $\E$.
Then the soft closure of $(G,\E)$ respect to the soft subspace $(Y,\Tau_Y,\E)$
coincides with the soft intersection of its soft closure respect to the soft space $(X,\Tau,\E)$
and of the absolute soft set $\absolutesoftset[Y]$ of the subspace, that is
$$\softcl[Y]{(G,\E)} \, \softequal \, \softcl{(G,\E)} \, \softcap \, \absolutesoftset[Y] . $$
\end{proposition}

%----------------------------

%-- definition of soft continuous mapping
\begin{definition}{\rm\cite{zorlutuna}}
\label{def:softcontinuousmapping}
Let $\varphi_\psi : \SSG{X}{\E} \to \SSG{X'}{\E'}$ be a soft mapping
between two soft topological spaces $(X,\Tau,\E)$ and $(X',\Tau',\E')$
induced by the mappings $\varphi: X \to X'$ and $\psi: \E \to \E'$
and $(x_\alpha, \E) \in \SPE[X]$ be a soft point over $X$.
We say that the soft mapping $\varphi_\psi$ is \df{soft continuous at the soft point $(x_\alpha, \E)$}
if for each soft neighbourhood $(G,\E')$
of $\varphi_\psi (x_\alpha, \E) $ in $(X',\Tau',\E')$
there exists some soft neighbourhood $(F,\E)$ of $(x_\alpha, \E)$ in $(X,\Tau,\E)$
such that $\varphi_\psi (F,\E) \, \softsubseteq \, (G,\E')$.
\\
If $\varphi_\psi$ is soft continuous at every soft point $(x_\alpha, \E) \in \SPE[X]$,
then $\varphi_\psi : \SSG{X}{\E} \to \SSG{X'}{\E'}$  is called \df{soft continuous} on $X$.
\end{definition}

%-- characterization of soft continuous mapping by soft inverse images of soft open sets
\begin{proposition}{\rm\cite{zorlutuna}}
\label{pro:characterizationofsoftcontinuitybysoftopensets}
Let $\varphi_\psi : \SSG{X}{\E} \to \SSG{X'}{\E'}$ be a soft mapping
between two soft topological spaces $(X,\Tau,\E)$ and $(X',\Tau',\E')$
induced by the mappings $\varphi: X \to X'$ and $\psi: \E \to \E'$.
Then the soft mapping $\varphi_\psi$ is soft continuous if and only if
every soft inverse image of a soft open set in $X'$ is a soft open set in $X$,
that is, if for each $(G,\E') \in \Tau'$
we have that $\varphi_\psi^{-1} (G,\E') \in \Tau$.
\end{proposition}

%-- characterization of soft continuous mapping by soft inverse images of soft closed sets
\begin{proposition}{\rm\cite{zorlutuna}}
\label{pro:characterizationofsoftcontinuitybysoftclosedsets}
Let $\varphi_\psi : \SSG{X}{\E} \to \SSG{X'}{\E'}$ be a soft mapping
between two soft topological spaces $(X,\Tau,\E)$ and $(X',\Tau',\E')$
induced by the mappings $\varphi: X \to X'$ and $\psi: \E \to \E'$.
Then the soft mapping $\varphi_\psi$ is soft continuous if and only if
every soft inverse image of a soft closed set in $X'$ is a soft closed set in $X$,
that is, if for each $(C,\E') \in \sigma(X',\E')$
we have that $\varphi_\psi^{-1} (C,\E') \in \sigma(X,\E)$.
\end{proposition}

%-- restriction of a soft mapping
\begin{definition}{\rm\cite{zorlutuna}}
\label{def:restrictionofsoftmapping}
Let $\varphi_\psi : \SSG{X}{\E} \to \SSG{X'}{\E'}$ be a soft mapping
between two soft topological spaces $(X,\Tau,\E)$ and $(X',\Tau',\E')$
induced by the mappings $\varphi: X \to X'$ and $\psi: \E \to \E'$,
and let $Y$ be a nonempty subset of $X$,
the \df{restriction} of the soft mapping $\varphi_\psi$ to $Y$,
denoted by ${\varphi_\psi}_{|Y}$,
is the soft mapping ${(\varphi_{|Y})}_\psi : \SSG{Y}{\E} \to \SSG{X'}{\E'}$
induced by the restriction $\varphi_{|Y}: Y \to X'$ of the mapping $\varphi$ between the universe sets
and by the same mapping $\psi: \E \to \E'$ between the parameter sets.
\end{definition}

%-- soft continuity of the restriction of a soft mapping
\begin{proposition}{\rm\cite{zorlutuna}}
\label{pro:softcontinuousrestriction}
If $\varphi_\psi : \SSG{X}{\E} \to \SSG{X'}{\E'}$ is a soft continuous mapping
between two soft topological spaces $(X,\Tau,\E)$ and $(X',\Tau',\E')$,
then its restriction ${\varphi_\psi}_{|Y} : \SSG{Y}{\E} \to \SSG{X'}{\E'}$
to a nonempty subset $Y$ of $X$ is soft continuous too.
\end{proposition}

%-------------------

%-- soft continuity of the corestriction of a soft mapping
\begin{proposition}
\label{pro:softcontinuouscorestriction}
If $\varphi_\psi : \SSG{X}{\E} \to \SSG{X'}{\E'}$ is a soft continuous mapping
between two soft topological spaces $(X,\Tau,\E)$ and $(X',\Tau',\E')$,
then its corestriction $\varphi_\psi : \SSG{X}{\E} \to \varphi_\psi \left( \SSG{X}{\E} \right)$
is soft continuous too.
\end{proposition}
\begin{proof}
It easily follows from Definitions \ref{def:softmapping} and \ref{def:softinverseimage},
and Proposition \ref{pro:characterizationofsoftcontinuitybysoftopensets}.
\end{proof}

%-------------------------------------------------------

%-- definition of soft subbase
\begin{definition}{\rm\cite{aygunoglu}}
\label{def:softsubbase}
Let $(X,\Tau,\E)$ be a soft topological space over $X$ and $\Scal \subseteq \Tau$ be
a non-empty subset of soft open sets.
We say that $\Scal$ is a \df{soft open subbase} for $(X,\Tau,\E)$ if the family of all finite soft intersections
of members of $\Scal$ forms a soft open base for $(X,\Tau,\E)$.
\end{definition}

%-- soft topology generated by a soft subbase
\begin{proposition}{\rm\cite{aygunoglu}}
\label{pro:softtopologyinducedbyasoftsubbase}
Let $\Scal \subseteq \SSG{X}{\E}$ be a family of soft sets over $X$, containing both
the null soft set $\nullsoftset$ and the absolute soft set $\absolutesoftset[X]$.
Then the family $\Tau(\Scal)$
of all soft union of finite soft intersections of soft sets in $\Scal$
is a soft topology having $\Scal$ as soft open subbase.
\end{proposition}

%-- soft topology generated by a family of soft sets
\begin{definition}{\rm\cite{aygunoglu}}
\label{def:softtopologyinducedbyasoftsubbase}
Let $\Scal \subseteq \SSG{X}{\E}$ be a a family of soft sets over $X$ respect to a set of parameters $\E$
and such that $\nullsoftset, \absolutesoftset[X] \in \Scal$, then
the soft topology $\Tau(\Scal)$ of the above Proposition \ref{pro:softtopologyinducedbyasoftsubbase}
is called the \df{soft topology generated} by the soft open subbase $\Scal$ over $X$
and $(X,\Tau(\Scal), \E)$ is said to be the \df{soft topological space generated by $\Scal$} over $X$.
\end{definition}

%-- definition of initial soft topology
\begin{definition}{\rm\cite{aygunoglu}}
\label{def:initialsofttopology}
Let $\SSG{X}{\E}$ be the set of all the soft sets over a universe set $X$
with respect to a set of parameter $\E$
and consider a family of soft topological spaces $\left\{ (Y_i, \Tau_i, \E_i ) \right\}_{i \in I}$
and a corresponding family $\left\{ (\varphi_\psi)_i \right\}_{i\in I}$
of soft mappings $(\varphi_\psi)_i = (\varphi_i)_{\psi_i} : \SSG{X}{\E} \to \SSG{Y_i}{\E_i}$
induced by the mappings $\varphi_i: X \to Y_i$ and $\psi_i: \E \to \E_i$ (with $i \in I$).
Then the soft topology $\Tau(\Scal)$ generated by the soft open subbase
$\Scal = \left\{ (\varphi_\psi)_i^{-1}(G,\E_i) : \, (G,\E_i) \in \Tau_i , \, i \in I \right\}$
of all soft inverse images of soft open sets of $\Tau_i$
under the soft mappings $(\varphi_\psi)_i$
is called the \df{initial soft topology}
induced on $X$ by the family of soft mappings $\left\{ (\varphi_\psi)_i \right\}_{i\in I}$
and it is denoted by $\Tau_{ini}\!\left(X, \E, Y_i, \E_i, (\varphi_\psi)_i; i\in I \right)$.
\end{definition}

%-- characterization of initial soft topology
\begin{proposition}{\rm\cite{aygunoglu}}
\label{pro:characterizationofsoftinitialtopology}
The initial soft topology $\Tau_{ini}\!\left(X, \E, Y_i, \E_i, (\varphi_\psi)_i; i\in I \right)$
induced on $X$ by the family of soft mappings $\left\{ (\varphi_\psi)_i \right\}_{i\in I}$
is the coarsest soft topology on $\SSG{X}{\E}$ for which
all the soft mappings $(\varphi_\psi)_i : \SSG{X}{\E} \to \SSG{Y_i}{\E_i}$ (with $i \in I$)
are soft continuous.
\end{proposition}

%-- definition of soft projection mappings
\begin{definition}{\rm\cite{aygunoglu}}
\label{def:softprojectionmappings}
Let $\left\{ (X_i, \Tau_i, \E_i ) \right\}_{i \in I}$ be a family of soft topological spaces
over the universe sets $X_i$ with respect to the sets of parameters $\E_i$, respectively.
For every $i \in I$, the soft mapping
${\left(\pi_i \right)}_{\rho_i} : \SSG{\prod_{i \in I} X_i}{\prod_{i \in I} \E_i} \to \SSG{X_i}{\E_i}$
induced by the canonical projections $\pi_i : \prod_{i \in I} X_i \to X_i$
and $\rho_i : \prod_{i \in I} \E_i \to \E_i$
is said the \df{$i$-th soft projection mapping}
and, by setting $(\pi_\rho)_i = {\left(\pi_i \right)}_{\rho_i}$, it will be denoted by
$(\pi_\rho)_i : \SSG{\prod_{i \in I} X_i}{\prod_{i \in I} \E_i} \to \SSG{X_i}{\E_i}$.
\end{definition}

%-- definition of soft topological product space
\begin{definition}{\rm\cite{aygunoglu}}
\label{def:softtopologicalproduct}
Let $\left\{ (X_i, \Tau_i, \E_i ) \right\}_{i \in I}$ be a family of soft topological spaces
and let $\left\{ (\pi_\rho)_i \right\}_{i \in I}$ be the corresponding family of soft projection mappings
$(\pi_\rho)_i : \SSG{\prod_{i \in I} X_i}{\prod_{i \in I} \E_i} \to \SSG{X_i}{\E_i}$
(with $i \in I$).
Then, the initial soft topology
$\Tau_{ini}\!\left( \prod_{i \in I} X_i, \E, X_i, \E_i, (\pi_\rho)_i; i\in I \right)$
induced on $\prod_{i \in I} X_i$ by the family of soft projection mappings
$\left\{ (\pi_\rho)_i \right\}_{i \in I}$
is called the \df{soft product topology} of the soft topologies
$\Tau_i$ (with $i \in I$)
and denoted by $\Tau\!\!\left( \prod_{i \in I} X_i \right)$.
\\
The triplet $\left( \prod_{i \in I} X_i , \Tau\!\!\left( \prod_{i \in I} X_i \right), \prod_{i \in I} \E_i \right)$
will be said the \df{soft topological product space}
of the soft topological spaces $( X_i, \Tau_i, \E_i )$.
%, with $i \in I$.
\end{definition}

%----------------------

The following statement easily derives from Definition \ref{def:softtopologicalproduct}
and Proposition \ref{pro:characterizationofsoftinitialtopology}.

%-- characterization of soft topological product
\begin{corollary}
\label{cor:characterizationofsofttopologicalproduct}
The soft product topology $\Tau\!\!\left( \prod_{i \in I} X_i \right)$
is the coarsest soft topology over $\SSG{\prod_{i \in I} X_i}{\prod_{i \in I} \E_i}$ for which
all the soft projection mappings
$(\pi_\rho)_i : \SSG{\prod_{i \in I} X_i}{\prod_{i \in I} \E_i} \to \SSG{X_i}{\E_i}$
(with $i \in I$) are soft continuous.
\end{corollary}

%-------------

%-- definition of soft slab
\begin{definition}{\rm\cite{nordo2019prod}}
\label{def:softslab}
Let $\left( \prod_{i \in I} X_i , \Tau\!\!\left( \prod_{i \in I} X_i \right), \prod_{i \in I} \E_i \right)$
be the soft topological product space
of the soft topological spaces $( X_i, \Tau_i, \E_i )$ (with $i \in I$)
and let $(\pi_\rho)_i : \SSG{\prod_{i \in I} X_i}{\prod_{i \in I} \E_i} \to \SSG{X_i}{\E_i}$
be the $i$-th the soft projection mapping.
The inverse soft image of a soft open set $(F_i,\E_i) \in \Tau_i$
under the soft projection mapping $(\pi_\rho)_i$, that is $(\pi_\rho)_i^{-1}  (F_i,\E_i)$
is called a \df{soft slab} and it is denoted by $\langle (F_i, \E_i) \rangle$.
\end{definition}

%--------------

Definitions \ref{def:softsubbase}, \ref{def:initialsofttopology},
\ref{def:softtopologicalproduct} and \ref{def:softslab}
give immediately the following property.

%-- subbase of the soft topological product
\begin{proposition}{\rm\cite{nordo2019prod}}
\label{pro:softopensubbaseofsofttopologicalproduct}
The family
$\Scal = \left\{ \langle (F_i, \E_i) \rangle : \, (F_i,\E_i) \in \Tau_i , \, i \in I \right\}$
of all soft slabs of soft open sets of $\Tau_i$
is a soft open subbase of the soft topological product space
$\left( \prod_{i \in I} X_i , \Tau\!\!\left( \prod_{i \in I} X_i \right), \prod_{i \in I} \E_i \right)$.
\end{proposition}

%---------------

%-- soft slab as soft cartesian product with only one non-trivial component
\begin{proposition}{\rm\cite{nordo2019prod}}
\label{pro:softslabassoftproduct}
Let $\left( \prod_{i \in I} X_i , \Tau\!\!\left( \prod_{i \in I} X_i \right), \prod_{i \in I} \E_i \right)$
be the soft topological product space of the soft topological spaces $( X_i, \Tau_i, \E_i )$, with $i \in I$
and let $(F_j,\E_j) \in \Tau_j$ be a soft open set of $X_j$, then its soft slab
$\langle (F_j,\E_j) \rangle$ coincides with a soft cartesian product in which only the $j$-th component
is the soft set $(F_j,\E_j)$ and the other ones are the absolute soft sets $\absolutesoftsetG{X_i}{\E_i}$, that is
$$\langle (F_j,\E_j) \rangle \, \softequal \, \softprod_{i \in I} (A_i, \E_i)
\quad \text{where} \quad
(A_i, \E_i) = \left\{
\begin{array}{ll}
(F_j, \E_j)
& \text{ if } i=j
\\[2mm]
\absolutesoftsetG{X_i}{\E_i} & \text{ otherwise}
\end{array}
\right. .
$$
\end{proposition}
\begin{proof}
By Definitions \ref{def:softslab} and \ref{def:softinverseimage}, we have that
$$\langle (F_j,\E_j) \rangle \, \softequal \, (\pi_\rho)_j^{-1}(F_j,\E_j) \, \softequal
\left( (\pi_\rho)_j^{-1}(F_j) , \prod_{i \in I} \E_i \right)$$
where $(\pi_\rho)_j^{-1}(F_j) : \prod_{i \in I} \E_i \to \PP(\prod_{i \in I} X_i)$ is the set-valued mapping
defined by
$\left((\pi_\rho)_j^{-1}(F_j)\right)\!(e) = \pi_j^{-1}\left( F_j\left( \rho_j(e) \right) \right)$
for every $e = \langle e_i \rangle_{i \in I} \in  \prod_{i \in I} \E_i$.

On the other hand, by Definition \ref{def:softproductofsoftsets}, it results
$$\softprod_{i \in I} (A_i, \E_i) \, \softequal
\left( \prod_{i \in I} A_i , \prod_{i \in I} \E_i \right)$$
where $\prod_{i \in I} A_i : \prod_{i \in I} \E_i \to \PP(\prod_{i \in I} X_i)$ is the set-valued mapping
defined by
$\left( \prod_{i \in I} A_i \right)\!(e) = \prod_{i \in I} A_i(e_i)$
for every $e = \langle e_i \rangle_{i \in I} \in  \prod_{i \in I} \E_i$
and since $A_j(e_j)=F_j(e_j)$ and $A_i(e_i)=X_i$ for every $i \in I \setminus \{ j \}$, it follows that
$\left( \prod_{i \in I} A_i \right)\!(e) = \langle F_j(e_j) \rangle$,
where the last set is the classical slab of the set $F_j(e_j)$ in the usual cartesian product $\prod_{i \in I} X_i$.
Thus, we also have that
$\left( \prod_{i \in I} A_i \right)\!(e) = \pi_j^{-1}(F_j(e_j)) = \pi_j^{-1}(F_j(\rho_j(e)))
= \left( (\pi_\rho)_j^{-1}(F_j)\right)\!(e)$
for every $e \in \E$, and so, by Remark \ref{rem:softequal}, the soft equality holds.
\end{proof}

%------------------------

%-- definition of soft n-slab
\begin{definition}{\rm\cite{nordo2019prod}}
\label{def:softnslab}
The soft intersection of a finite family of slab $\langle (F_{i_1}, \E_{i_1}) \rangle$
of soft open sets $(F_{i_k}, \E_{i_k}) \in \Tau_{i_k}$ (with $k=1,\ldots n$), that is
$\softbigcap_{k=1}^n \langle (F_{i_k}, \E_{i_k}) \rangle$ is said to be a \df{soft $n$-slab}
and it is denoted by $\langle (F_{i_1}, \E_{i_1}), \ldots (F_{i_n}, \E_{i_n}) \rangle$.
\end{definition}

%----------------

Definitions \ref{def:softbase}, \ref{def:softsubbase}, \ref{def:initialsofttopology},
\ref{def:softtopologicalproduct} and \ref{def:softnslab}
allow us to obtain the following property.

%-- base of the soft topological product
\begin{proposition}{\rm\cite{nordo2019prod}}
\label{pro:softopenbaseofsofttopologicalproduct}
The family
$$\Bcal = \left\{ \langle (F_{i_1}, \E_{i_1}), \ldots (F_{i_n}, \E_{i_n}) \rangle : \,
(F_{i_k},\E_{i_k}) \in \Tau_{i_k} , \, i_k \in I , n \in \NN^* \right\}$$
of all soft $n$-slabs of soft open sets of $\Tau_i$
is a soft open base of the soft topological product space
$\left( \prod_{i \in I} X_i , \Tau\!\!\left( \prod_{i \in I} X_i \right), \prod_{i \in I} \E_i \right)$.
\end{proposition}

%-- soft n-slab as soft cartesian product with only a finite number of non-trivial components
\begin{proposition}{\rm\cite{nordo2019prod}}
\label{pro:softnslabassoftproduct}
Let $\left( \prod_{i \in I} X_i , \Tau\!\!\left( \prod_{i \in I} X_i \right), \prod_{i \in I} \E_i \right)$
be the soft topological product space of the soft topological spaces $( X_i, \Tau_i, \E_i )$, with $i \in I$
and let $(F_{i_k},\E_{i_k}) \in \Tau_{i_k}$ be a finite family of soft open sets of $X_{i_k}$,
with $k=1,\ldots n$, respectively, then the soft $n$-slab
$\langle (F_{i_1}, \E_{i_1}), \ldots (F_{i_n}, \E_{i_n}) \rangle$
coincides with a soft cartesian product in which only the $i_k$-th components
(with $k=1,\ldots n$)
are the soft sets $(F_{i_k},\E_{i_k})$ and the other ones are the absolute soft sets $\absolutesoftsetG{X_i}{\E_i}$, that is
$$\langle (F_{i_1}, \E_{i_1}),  \ldots (F_{i_n}, \E_{i_n}) \rangle
\, \softequal \, \softprod_{i \in I} (A_i, \E_i)$$
where
$$
(A_i, \E_i) = \left\{
\begin{array}{ll}
(F_{i_k}, \E_{i_k})
& \text{ if } i=i_k \text{ for some } k=1,\ldots n
\\[2mm]
\absolutesoftsetG{X_i}{\E_i} & \text{ otherwise}
\end{array}
\right. .
$$
\end{proposition}
\begin{proof}
Similarly to the proof of Proposition \ref{pro:softslabassoftproduct},
by applying Definitions \ref{def:softnslab}, \ref{def:softintersection}, \ref{def:softslab}
and \ref{def:softinverseimage}, we have that
$$
\begin{array}{ll}
\langle (F_{i_1}, \E_{i_1}), \ldots (F_{i_n}, \E_{i_n}) \rangle \!\!\!\! &\softequal \;
\softbigcap_{k=1}^n \langle (F_{i_k}, \E_{i_k}) \rangle
\\[3mm]
&\softequal \; \softbigcap_{k=1}^n  (\pi_\rho)_{i_k}^{-1}(F_{i_k},\E_{i_k})
\\[3mm]
& \softequal \left( \bigcap_{k=1}^n (\pi_\rho)_{i_k}^{-1}(F_{i_k}) , \prod_{i \in I} \E_i \right)
\end{array}
$$
where $\bigcap_{k=1}^n (\pi_\rho)_{i_k}^{-1}(F_{i_k}) : \prod_{i \in I} \E_i \to \PP(\prod_{i \in I} X_i)$
is the set-valued mapping defined by
$\left( \bigcap_{k=1}^n (\pi_\rho)_{i_k}^{-1}(F_{i_k}) \right)\!(e)
= \bigcap_{k=1}^n  \pi_{i_k}^{-1}\left( F_{i_k}\left( \rho_{i_k}(e) \right) \right)$
for every $e = \langle e_i \rangle_{i \in I} \in  \prod_{i \in I} \E_i$.

On the other hand, by Definition \ref{def:softproductofsoftsets}, it results
$$\softprod_{i \in I} (A_i, \E_i) \, \softequal
\left( \prod_{i \in I} A_i , \prod_{i \in I} \E_i \right)$$
where $\prod_{i \in I} A_i : \prod_{i \in I} \E_i \to \PP(\prod_{i \in I} X_i)$ is the set-valued mapping
defined by
$\left( \prod_{i \in I} A_i \right)\!(e) = \prod_{i \in I} A_i(e_i)$
for every $e = \langle e_i \rangle_{i \in I} \in  \prod_{i \in I} \E_i$
and since $A_{i_k}(e_{i_k})=F_{i_k}(e_{i_k})$ for every $k=1,\ldots n$
and $A_i(e_i)=X_i$ for every $i \in I \setminus \{ i_1,\ldots i_n \}$, it follows that
$\left( \prod_{i \in I} A_i \right)\!(e) = \langle F_{i_1}(e_{i_1}), \ldots F_{i_n}(e_{i_n}) \rangle$,
where the last set is the classical $n$-slab of the sets $F_{i_k}(e_{i_k})$ (for $k=1,\ldots n$)
in the usual cartesian product $\prod_{i \in I} X_i$.
Thus, we also have that
$\left( \prod_{i \in I} A_i \right)\!(e)
= \bigcap_{k=1}^n \langle F_{i_k}(e_{i_k}) \rangle
= \bigcap_{k=1}^n \pi_{i_k}^{-1}(F_{i_k}(e_{i_k}))
= \bigcap_{k=1}^n \pi_{i_k}^{-1}(F_{i_k}(\rho_{i_k}(e)))
= \left( \bigcap_{k=1}^n (\pi_\rho)_{i_k}^{-1}(F_{i_k}) \right)\!(e)$
for every $e \in \E$, and so, by Remark \ref{rem:softequal}, the proposition is proved.
\end{proof}

%-----------------------------

%-- characterization of soft continuity on a soft topological product
\begin{proposition}{\rm\cite{aygunoglu}}
\label{pro:characterizationofsoftcontinuityonsoftotpologicalproduct}
Let $\left\{ (X_i, \Tau_i, \E_i ) \right\}_{i \in I}$ be a family of soft topological spaces,
$\left( X , \Tau\!(X), \E \right)$
be the soft topological product of such soft spaces
induced on the product $X=\prod_{i \in I} X_i $ of universe sets
with respect to the product $\E=\prod_{i \in I} \E_i $ of the sets of parameters,
$(Y, \Tau', \E')$ be a soft topological space and
$\varphi_\psi : \SSG{Y}{\E'} \to \SSG{X}{\E}$ be a soft mapping
induced by the mappings $\varphi: Y \to X$ and $\psi: \E' \to \E$.
Then the soft mappings $\varphi_\psi $ is soft continuous if and only if,
for every $i \in I$, the soft compositions  $(\pi_\rho)_i \, \softcirc \,
\varphi_\psi$
with the soft projection mappings $(\pi_\rho)_i : \SSG{X}{\E} \to \SSG{X_i}{\E_i}$
are soft continuous mappings.
\end{proposition}

%-------------------------

Let us note that the soft cartesian product $\softprod_{i \in I} (F_i,\E_i)$
of a family $\left\{(F_i,\E_i) \right\}_{i\in I}$ of soft sets
over a set $X_i$ with respect to a set of parameters $\E_i$, respectively,
as introduced in Definition \ref{def:softproductofsoftsets},
is a soft set of the soft topological product space
$\left( \prod_{i \in I} X_i  , \Tau\!\!\left( \prod_{i \in I} X_i \right), \prod_{i \in I} \E_i \right)$
i.e. that $\softprod_{i \in I} (F_i,\E_i) \in \SSG{\prod_{i \in I} X_i}{\prod_{i \in I} \E_i}$
and the following statement holds.

%---- soft closure of a soft product
\begin{proposition}{\rm\cite{nordo2019prod}}
\label{pro:softclosureofasoftproduct}
Let $\left( \prod_{i \in I} X_i  , \Tau\!\!\left( \prod_{i \in I} X_i \right), \prod_{i \in I} \E_i \right)$
be the soft topological product space
of a family $\left\{ (X_i, \Tau_i, \E_i ) \right\}_{i \in I}$ of soft topological spaces
and let $\softprod_{i \in I} (F_i,\E_i)$ be the soft product
in $\SSG{\prod_{i \in I} X_i}{\prod_{i \in I} \E_i}$
of a family $\left\{(F_i,\E_i) \right\}_{i\in I}$ of soft sets of $\SSG{X_i}{\E_i}$,
for every $i \in I$.
Then the soft closure of $\softprod_{i \in I} (F_i,\E_i)$
in the soft topological product
$\left( \prod_{i \in I} X_i  , \Tau\!\!\left( \prod_{i \in I} X_i \right), \prod_{i \in I} \E_i \right)$
coincides with the soft product of the corresponding soft closures of the soft sets $(F_i,\E_i)$
in the corresponding soft topological spaces $(X_i, \Tau_i, \E_i )$, that is:
$$
\softclpar[\prod_{i \in I} X_i]{ \softprod_{i \in I} (F_i,\E_i)} \softequal \,
\softprod_{i \in I} \softcl[X_i]{(F_i,\E_i)} .
$$
\end{proposition}
%------
\begin{proof}
Let $\left( x_\alpha , \prod_{i \in I} \E_i \right)$ be a soft point
of $\SPG{\prod_{i \in I} X_i}{\prod_{i \in I} \E_i}$,
with $x = \langle x_i \rangle_{i \in I}$ and $\alpha = \langle \alpha_i \rangle_{i \in I}$,
such that
$\left( x_\alpha , \prod_{i \in I} \E_i \right) \softin \,
\softclpar[\prod_{i \in I} X_i]{ \softprod_{i \in I} (F_i,\E_i)} $.
For any $j \in I$, let us consider a soft open set $(N_j, \E_j) \in \Tau_j$
such that $\left( (x_j)_{\alpha_j} , \E_j \right) \softin (N_j, \E_j)$.
By Proposition \ref{pro:softopensubbaseofsofttopologicalproduct}, the soft slab
$\langle (N_j, \E_j)\rangle$
is a soft open set of the soft open subbase of the soft topological product space $\prod_{i \in I} X_i$.
By Proposition \ref{pro:softslabassoftproduct}, we know that
$$\langle (N_j,\E_j) \rangle \, \softequal \, \softprod_{i \in I} (A_i, \E_i)
\quad \text{where} \quad
(A_i, \E_i) = \left\{
\begin{array}{ll}
(N_j, \E_j)
& \text{ if } i=j
\\[2mm]
\absolutesoftsetG{X_j}{\E_j} & \text{ otherwise}
\end{array}
\right.
$$
and so that $\left( x_\alpha , \prod_{i \in I} \E_i \right) \softin \langle (N_i, \E_i)\rangle$.
Thus, by our hypothesis, it follows that
$$\softprod_{i \in I} (F_i,\E_i) \, \softcap \, \softprod_{i \in I} (A_i, \E_i)
\, \softnotequal \nullsoftsetG{\prod_{i \in I} \E_i} $$
which, by Proposition \ref{pro:distributivepropertyofsoftproductrespecttosoftintersection}, is equivalent to
$$\softprod_{i \in I} \left( (F_i,\E_i) \softcap (A_i, \E_i) \right)
\, \softnotequal \nullsoftsetG{\prod_{i \in I} \E_i} $$
and hence, by Corollary \ref{cor:softnullityofsoftproduct}, it follows in particular that
$$(F_j,\E_j) \softcap (A_j, \E_j) \, \softnotequal \nullsoftsetG{E_j} $$
i.e.
$$(F_j,\E_j) \softcap (N_j, \E_j) \, \softnotequal \nullsoftsetG{E_j} .$$
Thus, by Definition \ref{def:softadherentpoint}, we have that
$\left( (x_j)_{\alpha_j} , \E_j \right)$ is a soft adherent point for the soft set $(F_j,\E_j)$
and so, by Proposition \ref{pro:softclosureisthesetofsoftadherentpoints}, that
$$\left( (x_j)_{\alpha_j} , \E_j \right) \softin \softcl[X_j]{(F_j,\E_j)}, \;\; \text{ for any fixed } j \in I$$
that, by Proposition \ref{pro:softpointsoftbelongingtosoftproduct}, is equivalent to say that
$$\left( x_\alpha , \prod_{i \in I} \E_i \right) \softin \,
\softprod_{i \in I} \softcl[X_i]{(F_i,\E_i)}$$
and, by using Proposition \ref{pro:softinclusionbysoftpoint}, this proves that
$$\softclpar[\prod_{i \in I} X_i]{ \softprod_{i \in I} (F_i,\E_i)} \softsubseteq \,
\softprod_{i \in I} \softcl[X_i]{(F_i,\E_i)} \, . $$

On the other hand, let $\left( x_\alpha , \prod_{i \in I} \E_i \right) \softin \,
\softprod_{i \in I} \softcl[X_i]{(F_i,\E_i)}$.
By Proposition \ref{pro:softpointsoftbelongingtosoftproduct}, we have that
$\left( (x_i)_{\alpha_i} , \E_i \right) \softin \softcl[X_i]{(F_i,\E_i)}$
for every $i \in I$.
Let us consider a soft open set $\left( N, \prod_{i \in I} \E_i \right)$ of $\prod_{i \in I} X_i$ such that
$\left( x_\alpha , \prod_{i \in I} \E_i \right) \softin \left( N, \prod_{i \in I} \E_i \right)$.
By Propositions \ref{pro:characterizationofsoftopenbase} and \ref{pro:softopenbaseofsofttopologicalproduct}
and Definition \ref{def:softbase},
we have that there exists a finite family of soft open sets $(N_{i_k}, \E_{i_k} )\in \Tau_{i_k}$
with $k=1,\ldots n$ and $n \in \NN^*$ such that
$$\left( x_\alpha , \prod_{i \in I} \E_i \right) \softin \,
\langle (N_{i_1}, \E_{i_1}), \ldots (N_{i_n}, \E_{i_n}) \rangle
\, \softsubseteq \left( N, \prod_{i \in I} \E_i \right) .$$
Since, by Proposition \ref{pro:softnslabassoftproduct}, we have that
$$\langle (N_{i_1}, \E_{i_1}), \ldots (N_{i_n}, \E_{i_n}) \rangle
\, \softequal \, \softprod_{i \in I} (A_i, \E_i)$$
where
$$
(A_i, \E_i) = \left\{
\begin{array}{ll}
(N_{i_k}, \E_{i_k})
& \text{ if } i=i_k \text{ for some } k=1,\ldots n
\\[2mm]
\absolutesoftsetG{X_i}{\E_i} & \text{ otherwise}
\end{array}
\right. ,
$$
it follows that
$$\left( x_\alpha , \prod_{i \in I} \E_i \right) \softin \,
\softprod_{i \in I} (A_i, \E_i)
\, \softsubseteq \left( N, \prod_{i \in I} \E_i \right) .$$
Now, we claim that
$$
\softprod_{i \in I} (A_i, \E_i) \, \softcap \, \softprod_{i \in I} (F_i, \E_i)
\, \softnotequal \nullsoftsetG{\prod_{i \in I} \E_i} .
$$
In fact, for every $k=1,\ldots n$, we have that $\left( (x_{i_k})_{\alpha_{i_k}} , \E_{i_k} \right) \in \Tau_{i_k}$
and so, being $\left( (x_{i_k})_{\alpha_{i_k}} , \E_{i_k} \right) \softin \softcl[X_{i_k}]{(F_{i_k},\E_{i_k})}$,
by Proposition \ref{pro:softclosureisthesetofsoftadherentpoints} and Definition \ref{def:softadherentpoint},
it follows that
$$\left( A_{i_k}, \E_{i_k} \right) \softcap \left( F_{i_k}, \E_{i_k} \right) \softequal
\left( N_{i_k}, \E_{i_k} \right) \softcap \left( F_{i_k}, \E_{i_k} \right)
\softnotequal \nullsoftsetG{E_{i_k}}$$
while, for every $i \in I \setminus \left\{ i_1, \ldots i_n \right\}$,
by Proposition \ref{pro:propertiesunionandintersection}(6), it trivially results
$$\left(A_i, \E_i \right) \softcap \left( F_i, \E_i \right) \softequal
\absolutesoftsetG{X_i}{\E_i} \softcap \left( F_i, \E_i \right) \softequal
\left( F_i, \E_i \right)
\softnotequal \nullsoftsetG{E_i}$$
and so the previous assertion follows from Proposition \ref{cor:softnullityofsoftproduct}.

Thus, a fortiori, we have that
$$
\left( N, \prod_{i \in I} \E_i \right) \softcap \;
\softprod_{i \in I} (F_i, \E_i)
\, \softnotequal \nullsoftsetG{\prod_{i \in I} \E_i}
$$
which, by Definition \ref{def:softadherentpoint} and Proposition \ref{pro:softclosureisthesetofsoftadherentpoints},
means that
$$\left( x_\alpha , \prod_{i \in I} \E_i \right) \softin \,
\softclpar[\prod_{i \in I} X_i]{ \softprod_{i \in I} (F_i,\E_i)} $$
and hence, by Proposition \ref{pro:softinclusionbysoftpoint}, we have
$$
\softprod_{i \in I} \softcl[X_i]{(F_i,\E_i)}
\, \softsubseteq \,
\softclpar[\prod_{i \in I} X_i]{ \softprod_{i \in I} (F_i,\E_i)}
$$
that concludes our proof.
\end{proof}

%-------------------------------------------------------------------------------------------
\section{Soft Embedding Lemma}

%- definition of soft homeomorphism
\begin{definition}{\rm\cite{aras}}
\label{def:softhomeomorphism}
Let $(X,\Tau,\E)$ and $(X',\Tau',\E')$ be two soft topological spaces
over the universe sets $X$ and $X'$ with respect to the sets of parameters $\E$ and $\E'$, respectively.
We say that a soft mapping $\varphi_\psi : \SSG{X}{\E} \to \SSG{X'}{\E'}$
is a \df{soft homeomorphism} if it is soft continuous, bijective
and its soft inverse mapping $\varphi_\psi^{-1} : \SSG{X'}{\E'} \to \SSG{X}{\E} $
is soft continuous too.
In such a case, the soft topological spaces $(X,\Tau,\E)$ and $(X',\Tau',\E')$
are said \df{soft homeomorphic} and we write that $(X,\Tau,\E) \softhomeomorphic (X',\Tau',\E')$.
\end{definition}

%- definition of soft embedding
\begin{definition}%{\rm\cite{aras}}
\label{def:softembedding}
Let $(X,\Tau,\E)$ and $(X',\Tau',\E')$ be two soft topological spaces.
We say that a soft mapping $\varphi_\psi : \SSG{X}{\E} \to \SSG{X'}{\E'}$
is a \df{soft embedding} if its corestriction
$\varphi_\psi : \SSG{X}{\E} \to \varphi_\psi \left( \SSG{X}{\E} \right)$
is a soft homeomorphism.
\end{definition}

%- definition of soft closed mapping
\begin{definition}{\rm\cite{aras}}
\label{def:softclosedmapping}
Let $(X,\Tau,\E)$ and $(X',\Tau',\E')$ be two soft topological spaces.
We say that a soft mapping $\varphi_\psi : \SSG{X}{\E} \to \SSG{X'}{\E'}$
is a \df{soft closed mapping} if the soft image of every soft closed set of $(X,\Tau,\E)$
is a soft closed set of $(X',\Tau',\E')$, that is if for any $(C,\E) \in \sigma(X,\E)$,
we have $\varphi_\psi (C,\E) \in \sigma(X',\E')$.
\end{definition}

%-- characterization of soft embedding as soft continuous, injective and soft closed mappings
\begin{proposition}%{\rm\cite{aras}} %----- NEW RESULT
\label{pro:characterizationsoftembedding}
Let $\varphi_\psi : \SSG{X}{\E} \to \SSG{X'}{\E'}$ be a soft mapping
between two soft topological spaces $(X,\Tau,\E)$ and $(X',\Tau',\E')$.
If $\varphi_\psi$ is a soft continuous, injective and soft closed mapping
then it is a soft embedding.
\end{proposition}
\begin{proof}
If we consider the soft mapping
$\varphi_\psi : \SSG{X}{\E} \to \varphi_\psi \left( \SSG{X}{\E} \right)$,
by hypothesis and Proposition \ref{pro:softcontinuouscorestriction},
it immediately follows that it is a soft continuous bijective mapping
and so we have only to prove that its soft inverse mapping
$\varphi_\psi^{-1} = \left( \varphi^{-1} \right)_{\psi^{-1}}
: \varphi_\psi \left( \SSG{X}{\E} \right) \to \SSG{X}{\E}$ is continuous too.
In fact, because the bijectiveness of the corestriction and Remark \ref{rem:imageosoftinversemapping},
for every soft closed set $(C,\E) \in \sigma(X,\E)$, the soft inverse image
of the $(C,\E)$ under the soft inverse mapping $\varphi_\psi^{-1}$
coincides with the soft image of the same soft set under the soft mapping $\varphi_\psi$,
that is $\left(\varphi_\psi^{-1}\right)^{-1}\!\!(C,\E) \, \softequal \, \varphi_\psi(C,\E)$ and
since by hypothesis $\varphi_\psi$ is soft closed, it follows that
$\left(\varphi_\psi^{-1}\right)^{-1}\!\!(C,\E) \in \sigma(X',\E')$
which, by Proposition \ref{pro:characterizationofsoftcontinuitybysoftclosedsets},
proves that $\varphi_\psi^{-1} : \SSG{X'}{\E'} \to \SSG{X}{\E}$
is a soft continuous mapping,
and so, by Proposition \ref{pro:softcontinuousrestriction},
we finally have that $\varphi_\psi^{-1} : \varphi_\psi \left( \SSG{X}{\E} \right) \to \SSG{X}{\E}$
is a soft continuous mapping.
\end{proof}

%---------------------

%-- definition of soft diagonal mapping
\begin{definition}
\label{def:softdiagonalmapping}
Let $(X,\Tau,\E)$ be a soft topological space over a universe set $X$
with respect to a set of parameter $\E$,
let $\left\{ (X_i, \Tau_i, \E_i ) \right\}_{i \in I}$ be a family of soft topological spaces
over a universe set $X_i$ with respect to a set of parameters $\E_i$, respectively
and consider a family $\left\{ (\varphi_\psi)_i \right\}_{i\in I}$
of soft mappings $(\varphi_\psi)_i = {\left( \varphi_i \right)}_{\psi_i}: \SSG{X}{\E} \to \SSG{X_i}{\E_i}$
induced by the mappings $\varphi_i : X \to X_i$ and $\psi_i : \E \to \E_i$ (with $i \in I$).
Then the soft mapping
$\Delta = \varphi_\psi : \SSG{X}{\E} \to \SSG{\prod_{i \in I} X_i}{\prod_{i \in I} \E_i} $
induced by the diagonal mappings (in the classical meaning)
$\varphi = \Delta_{i \in I} \varphi_i : X \to \prod_{i \in I} X_i$ on the universes sets
and $\psi  = \Delta_{i \in I} \psi_i: \E \to \prod_{i \in I} \E_i$ on the sets of parameters
(respectively defined by $\varphi(x) = \langle \varphi_i(x) \rangle_{i \in I}$ for every $x \in X$
and by $\psi(e) = \langle \psi_i(e) \rangle_{i \in I}$ for every $e \in \E$)
is called the \df{soft diagonal mapping} of the soft mappings $(\varphi_\psi)_i$ (with $i\in I$)
and it is denoted by
$\Delta = \Delta_{i \in I} (\varphi_\psi)_i : \SSG{X}{\E} \to \SSG{\prod_{i \in I} X_i}{\prod_{i \in I} \E_i} $.
\end{definition}

The following proposition establishes a useful relation about the soft image of a soft diagonal mapping.

\begin{proposition}%
{\rm\cite{nordo2019prod}}
\label{pro:softimageofasoftdiagonalmapping}
Let $(X,\Tau,\E)$ be a soft topological space over a universe set $X$ with respect to a set of parameter $\E$,
let $(F,\E) \in \SSG{X}{\E}$ be a soft set of $X$,
let $\left\{ (X_i, \Tau_i, \E_i ) \right\}_{i \in I}$ be a family of soft topological spaces
over a universe set $X_i$ with respect to a set of parameters $\E_i$, respectively and let
$\Delta = \Delta_{i \in I} (\varphi_\psi)_i : \SSG{X}{\E} \to \SSG{\prod_{i \in I} X_i}{\prod_{i \in I} \E_i} $
be the soft diagonal mapping of the soft mappings $(\varphi_\psi)_i$, with $i\in I$.
Then the soft image of the soft set $(F,\E)$ under the soft diagonal mapping $\Delta$ is soft contained
in the soft product of the soft images of the same soft set under the soft mappings $(\varphi_\psi)_i$, that is
$$\Delta (F,\E) \, \softsubseteq \, \softprod_{i\in I} (\varphi_\psi)_i(F,\E) .$$
\end{proposition}
\begin{proof}
Set $\varphi = \Delta_{i \in I} \varphi_i : X \to \prod_{i \in I} X_i$
and $\psi  = \Delta_{i \in I} \psi_i: \E \to \prod_{i \in I} \E_i$,
by Definition \ref{def:softdiagonalmapping},
we know that $\Delta = \Delta_{i \in I} (\varphi_\psi)_i =\varphi_\psi$.
%----
Suppose, by contradiction, that there exists some soft point $(x_\alpha, \E) \softin (F,\E)$
such that
$$\Delta (x_\alpha, \E) \, \softnotin \, \softprod_{i\in I} (\varphi_\psi)_i(F,\E) .$$
Set $\left( y_\beta, \prod_{i \in I} \E_i \right) \softequal \,
\Delta(x_\alpha, \E) \, \softequal \, \varphi_\psi (x_\alpha, \E)$,
by Proposition \ref{pro:softimageofasoftpoint}, it follows that
$$
\left( y_\beta, \prod_{i \in I} \E_i \right) \softequal
\left( \varphi(x)_{\psi(\alpha)}, \prod_{i \in I} \E_i \right)
$$
where
$$y = \langle y_i \rangle_{i \in I} = \varphi(x) = \left( \Delta_{i \in I} \varphi_i \right)(x)
= \langle \varphi_i(x) \rangle_{i \in I}$$
and
$$\beta = \langle \beta_i \rangle_{i \in I} = \psi(\alpha) = \left( \Delta_{i \in I} \psi_i \right)(\alpha)
= \langle \psi_i(\alpha) \rangle_{i \in I} \, . $$
So, set $(G_i,\E_i) \, \softequal \, (\varphi_\psi)_i(F,\E)$ for every $i \in I$, we have that
$$\left( y_\beta, \prod_{i \in I} \E_i \right)  \softnotin \, \softprod_{i\in I} (G_i,\E_i)$$
hence, by Proposition \ref{pro:softpointsoftbelongingtosoftproduct}, it follows that there exists some $j \in I$
such that
$$\left( (y_j)_{\beta_j} , \E_j \right) \softnotin \, (G_j, \E_j) $$
that, by Definition \ref{def:softpointsoftbelongstosoftset}, means
$$ y_j \notin G_j(\beta_j)$$
i.e.
$$ \varphi_j(x) \notin G_j\left( \psi_j(\alpha) \right)$$
and so, by using again Definition \ref{def:softpointsoftbelongstosoftset}, we have
$$\left( {\varphi_j(x)}_{\psi_j(\alpha)} , \E_j \right) \softnotin \, (G_j,\E_j)$$
that, by Proposition \ref{pro:softimageofasoftpoint}, is equivalent to
$$(\varphi_\psi)_j (x_\alpha, \E) \, \softnotin \, (G_j,\E_j)$$
which is a contradiction because we know that $(x_\alpha, \E) \softin (F,\E)$
and by Corollary \ref{cor:monotonicsoftimagesandsoftinverseimages}(1)
it follows $(\varphi_\psi)_j (x_\alpha, \E) \, \softin \, (\varphi_\psi)_j (F,\E)
\, \softequal \, (G_j,\E_j)$.
\end{proof}

%-------------------------------------------------------------

%-- definition of family of soft mapping separating points
\begin{definition}
\label{def:familyofsofttmappingseparatingsoftpoints}
Let $\left\{ (\varphi_\psi)_i \right\}_{i\in I}$ be a family of
soft mappings $(\varphi_\psi)_i : \SSG{X}{\E} \to \SSG{X_i}{\E_i}$
between a soft topological space $(X,\Tau,\E)$
and the members of a family of soft topological spaces
$\left\{ (X_i, \Tau_i, \E_i ) \right\}_{i \in I}$.
We say that the family $\left\{ (\varphi_\psi)_i \right\}_{i\in I}$
\df{soft separates soft points} of $(X,\Tau,\E)$
if for every $(x_\alpha, \E), (y_\beta, \E) \in \SPE[X]$ such that
$(x_\alpha, \E) \softnotequal (y_\alpha, \E)$ there exists some $j \in I$ such that
$(\varphi_\psi)_j (x_\alpha, \E) \softnotequal (\varphi_\psi)_j (y_\beta, \E)$.
\end{definition}

%-- definition of family of soft mapping separating soft points and soft closed sets
\begin{definition}
\label{def:familyofsofttmappingseparatingsoftpointsfromsoftclosedsets}
Let $\left\{ (\varphi_\psi)_i \right\}_{i\in I}$ be a family of
soft mappings $(\varphi_\psi)_i : \SSG{X}{\E} \to \SSG{X_i}{\E_i}$
between a soft topological space $(X,\Tau,\E)$
and the members of a family of soft topological spaces $\left\{ (X_i, \Tau_i, \E_i ) \right\}_{i \in I}$.
We say that the family $\left\{ (\varphi_\psi)_i \right\}_{i\in I}$
\df{soft separates soft points from soft closed sets} of $(X,\Tau,\E)$
if for every $(C,\E) \in \sigma(X,\E)$
and every $(x_\alpha, \E) \in \SPE[X]$ such that
$(x_\alpha, \E) \softin \absolutesoftsetG{X}{\E} \softsetminus (C,\E)$
there exists some $j \in I$ such that
$(\varphi_\psi)_j (x_\alpha, \E) \softnotin \softclpar[X_j]{(\varphi_\psi)_j (C,\E) }$.
\end{definition}

%-- soft embedding lemma
\begin{proposition}[\textbf{Soft Embedding Lemma}]
\label{pro:softembeddinglemma}
Let $(X,\Tau,\E)$ be a soft topological space,
$\left\{ (X_i, \Tau_i, \E_i ) \right\}_{i \in I}$ be a family of soft topological spaces
and $\left\{ (\varphi_\psi)_i \right\}_{i\in I}$ be a family
of soft continuous mappings $(\varphi_\psi)_i : \SSG{X}{\E} \to \SSG{X_i}{\E_i}$
that separates both the soft points
and the soft points from the soft closed sets of $(X,\Tau,\E)$.
Then the soft diagonal mapping $\Delta = \Delta_{i \in I} (\varphi_\psi)_i : \SSG{X}{\E}
\to \SSG{\prod_{i \in I} X_i}{\prod_{i \in I} \E_i} $
of the soft mappings $(\varphi_\psi)_i$ is a soft embedding.
\end{proposition}
%------------------------
\begin{proof}
%--- continuity
Let $\varphi = \Delta_{i \in I} \varphi_i$, $\psi  = \Delta_{i \in I} \psi_i$
and $\Delta = \Delta_{i \in I} (\varphi_\psi)_i = \varphi_\psi$ as in Definition \ref{def:softdiagonalmapping},
for every $i\in I$, by using Definition \ref{def:softcompositionofsoftmappings},
we have that every corresponding soft composition is given by
$$(\pi_\rho)_i \, \softcirc \, \Delta
= \left( (\pi_i)_{\rho_i} \right) \, \softcirc \, \varphi_\psi
= \left( \pi_i \circ \varphi \right)_{\rho_i \circ \psi}
= \left( \varphi_i \right)_{\psi_i}
 = (\varphi_\psi)_i$$
which, by hypothesis, is a soft continuous mapping.
%(see the Figure \ref{fig:commutativediagramsofdiagonaltmapping}).
Hence, by Proposition \ref{pro:characterizationofsoftcontinuityonsoftotpologicalproduct},
it follows that the soft diagonal mapping
$\Delta : \SSG{X}{\E} \to \SSG{\prod_{i \in I} X_i}{\prod_{i \in I} \E_i} $
is a soft continuous mapping.

%--- injectivity
Now, let $(x_\alpha, \E)$ and $(y_\beta, \E)$ be two distinct soft points of $\SPE[X]$.
Since, by hypothesis, the family $\left\{ (\varphi_\psi)_i \right\}_{i\in I}$
of soft mappings soft separates soft points,
by Definition \ref{def:familyofsofttmappingseparatingsoftpoints}, we have that
there exists some $j \in I$ such that
$(\varphi_\psi)_j (x_\alpha, \E) \softnotequal (\varphi_\psi)_j (y_\beta, \E)$,
that is
$$
\left( \varphi_j \right)_{\psi_j} (x_\alpha, \E) \, \softnotequal \,
\left( \varphi_j \right)_{\psi_j} (y_\beta, \E) .
$$
Hence, by Proposition \ref{pro:softimageofasoftpoint}, we have that:
$$
\left( {\varphi_j(x)}_{\psi_i(\alpha)}, \E_j \right) \softnotequal
\left( {\varphi_j(y)}_{\psi_i(\beta)}, \E_j \right)
$$
and so, by the Definition \ref{def:distinctssoftpoints} of distinct soft points,
it necessarily follows that:
$$
\varphi_j(x) \ne \varphi_j(y) \quad \text{ or } \quad \psi_j(\alpha) \ne \psi_j(\beta) .
$$
Since $\varphi = \Delta_{i \in I} \varphi_i : X \to \prod_{i \in I} X_i$
and $\psi  = \Delta_{i \in I} \psi_i: \E \to \prod_{i \in I} \E_i$ are usual diagonal mappings,
we have that:
$$
\varphi(x) \ne \varphi(y) \quad \text{ or } \quad \psi(\alpha) \ne \psi(\beta) \;
$$
and, by Definition \ref{def:distinctssoftpoints}, it follows that:
$$
\left( {\varphi(x)}_{\psi(\alpha)} , \prod_{i \in I} \E_i \right) \softnotequal
\left( {\varphi(y)}_{\psi(\beta)} , \prod_{i \in I} \E_i \right)
$$
hence, applying again Proposition \ref{pro:softimageofasoftpoint}, we get:
$$
\varphi_\psi (x_\alpha, \E) \, \softnotequal \,
\varphi_\psi (y_\beta, \E)
$$
that is:
$$
\Delta_{i \in I} (\varphi_\psi)_i (x_\alpha, \E) \, \softnotequal \,
\Delta_{i \in I} (\varphi_\psi)_i (y_\beta, \E)
$$
i.e. that $\Delta (x_\alpha, \E) \, \softnotequal \, \Delta (y_\beta, \E)$
which, by Corollary \ref{cor:injectivityofasoftmapping},
proves the injectivity of the soft diagonal mapping
$\Delta : \SSG{X}{\E} \to \SSG{\prod_{i \in I} X_i}{\prod_{i \in I} \E_i} $.

%--- soft closed
Finally, let $(C,\E) \in \sigma(X,\E)$ be a soft closed set in $X$
and, in order to prove that the soft image $\Delta(C,\E)$
is a soft closed set of $\sigma \! \left( \prod_{i \in I} X_i , \prod_{i \in I} \E_i \right)$,
consider a soft point  $(x_\alpha, \E) \in \SPE[X]$ such that
$\Delta (x_\alpha, \E) \softnotin \Delta(C,\E)$ and, hence,
by Corollary \ref{cor:monotonicsoftimagesandsoftinverseimages}(1), such that $(x_\alpha, \E) \softnotin (C,\E)$.
Since, by hypothesis, the family $\left\{ (\varphi_\psi)_i \right\}_{i\in I}$
of soft mappings soft separates soft points from soft closed sets,
by Definition \ref{def:familyofsofttmappingseparatingsoftpointsfromsoftclosedsets},
we have that there exists some $j \in I$ such that
$(\varphi_\psi)_j (x_\alpha, \E) \, \softnotin \, \softclpar[X_j]{(\varphi_\psi)_j (C,\E) }$,
that is:
$$
(\varphi_j)_{\psi_j} (x_\alpha, \E) \, \softnotin \, \softclpar[X_j]{(\varphi_\psi)_j (C,\E) }
$$
that, by Proposition \ref{pro:softimageofasoftpoint}, corresponds to:
$$
\left( {\varphi_j(x)}_{\psi_j(\alpha)}, \E_j \right) \, \softnotin \,
\softclpar[X_j]{(\varphi_\psi)_j (C,\E) } \, .
$$

So, set $(C_i, \E_i) \, \softequal \, \softclpar[X_i]{(\varphi_\psi)_i (C,\E)} $ for every $i \in I$, we have
in particular for $i=j$ that
$$
\left( {\varphi_j(x)}_{\psi_j(\alpha)}, \E_j \right) \softnotin \, (C_j, \E_j)
$$
which, by Definition \ref{def:softpointsoftbelongstosoftset}, is equivalent to say that:
$$
\varphi_j(x) \notin C_j \left( \psi_j(\alpha) \right)
$$
and since the diagonal mapping $\varphi = \Delta_{i \in I} \varphi_i : X \to \prod_{i \in I} X_i$
on the universes sets is defined by $\varphi(x) = \langle \varphi_i(x) \rangle_{i \in I}$, it follows that:
$$
\varphi(x) \notin \prod_{i \in I} C_i \left( \psi_i(\alpha) \right) .
$$
Now, since the diagonal mapping $\psi = \Delta_{i \in I} \psi_i : X \to \prod_{i \in I} X_i$
on the sets of parameters is defined by
$\psi(\alpha) = \Delta_{i \in I} \psi_i (\alpha) = \langle \psi_i(\alpha) \rangle_{i \in I}$,
using Definition \ref{def:softproductofsoftsets}, we obtain:
$$
\prod_{i \in I} C_i \left( \psi_i(\alpha) \right)
= \left( \prod_{i \in I} C_i \right) \left( \psi(\alpha) \right)
$$
and hence that
$$\varphi(x) \notin \left( \prod_{i \in I} C_i \right) \left( \psi(\alpha) \right)$$
which, by Definitions \ref{def:softpointsoftbelongstosoftset}
and \ref{def:softproductofsoftsets}, is equivalent to say that:
$$\left( \varphi(x)_{\psi(\alpha)}, \prod_{i \in I} \E_i \right) \softnotin \, \softprod_{i \in I} (C_i, \E_i)$$
that, by Proposition \ref{pro:softimageofasoftpoint}, means:
$$\varphi_\psi (x_\alpha, \E) \, \softnotin \,\softprod_{i \in I} (C_i, \E_i)$$
i.e.
$$\Delta (x_\alpha, \E) \, \softnotin \, \softprod_{i \in I} \softclpar[X_i]{(\varphi_\psi)_i (C,\E)} .$$
So, recalling, by Proposition \ref{pro:softclosureofasoftproduct}, that
$$\softclpar[\prod_{i \in I} X_i]{ \softprod_{i \in I} (\varphi_\psi)_i (C,\E)}
\softequal \, \softprod_{i \in I} \softclpar[X_i]{(\varphi_\psi)_i (C,\E)}$$
it follows that:
$$\Delta (x_\alpha, \E) \, \softnotin \, \softclpar[\prod_{i \in I} X_i]{ \softprod_{i \in I} (\varphi_\psi)_i (C,\E)} . $$
Since, by Propositions \ref{pro:softimageofasoftdiagonalmapping} and
\ref{pro:propertiesofsoftclosure}(3) we have
$$\Delta (C,\E) \, \softsubseteq \, \softprod_{i\in I} (\varphi_\psi)_i(C,\E)
\, \softsubseteq \,
\softclpar[\prod_{i \in I} X_i]{\softprod_{i\in I} (\varphi_\psi)_i(C,\E)}
$$
and, by applying Propositions \ref{pro:sofsoftclosureofoperators}(1)
and \ref{pro:propertiesofsoftclosure}(5), we obtain
$$
\softclpar[\prod_{i \in I} X_i]{\Delta (C,\E)}  \, \softsubseteq \,
\softclpar[\prod_{i \in I} X_i]{\softprod_{i\in I} (\varphi_\psi)_i(C,\E)}
$$
it follows, a fortiori, that
$$\Delta (x_\alpha, \E) \, \softnotin \, \softclpar[\prod_{i \in I} X_i]{\Delta (C,\E)} .$$
So, it is proved by contradiction that
$\softclpar[\prod_{i \in I} X_i]{\Delta (C,\E)} \softsubseteq \, \Delta (C,\E)$
and hence, by Proposition \ref{pro:propertiesofsoftclosure}(4)
and Definition \ref{def:softclosedmapping}, that
$\Delta : \SSG{X}{\E} \to \SSG{\prod_{i \in I} X_i}{\prod_{i \in I} \E_i} $
is a soft closed mapping.

%-- conclusion
Thus, we finally have that the soft diagonal mapping
$\Delta = \Delta_{i \in I} (\varphi_\psi)_i : \SSG{X}{\E} \to \SSG{\prod_{i \in I} X_i}{\prod_{i \in I} \E_i} $
is a soft continuous, injective and soft closed mapping
and so, by Proposition \ref{pro:characterizationsoftembedding},
it is a soft embedding.
\end{proof}

%-------------------------------------------------------------------------------------------

\section{Conclusion}
In this paper we have introduced the notions of family of soft mappings
separating points and points from closed sets and that of soft diagonal mapping
and we have proved a generalization to soft topological spaces of the well-known Embedding Lemma
for classical (crisp) topological spaces.
Such a result could be the start point for extending and investigating other important topics
such as extension and compactifications theorems, metrization theorems etc. in the context of soft topology.

%-------------------------------------------------------------------------------------------
%\section*{References}

\vskip 2cm

\parskip=0pt {\noindent {\sc Giorgio NORDO \newline
MIFT - Dipartimento di Scienze Matematiche e Informatiche, scienze Fisiche e scienze della Terra -- Universit\`{a} di Messina \newline
Viale F. Stagno D'Alcontres, 31 --
Contrada Papardo, salita Sperone, 98166 Sant'Agata -- Messina
(ITALY)}} \vskip 1pt \noindent
E-mail:  {\tt giorgio.nordo@unime.it}

\end{document}